\documentclass[a4paper,11pt]{amsart}
\usepackage{amsfonts}
\usepackage{amssymb}
\usepackage[utf8]{inputenc}
\usepackage{amsmath}
\usepackage{graphicx}
\usepackage[pdftex]{hyperref}
\providecommand{\U}[1]{\protect\rule{.1in}{.1in}}
\newtheorem{theorem}{Theorem}[section]
\newtheorem{proposition}[theorem]{Proposition}
\newtheorem{lemma}[theorem]{Lemma}

\newtheorem{remark}[theorem]{Remark}

\newcommand{\remove}[1]{ }

\def\be{\begin{equation}}
\def\ee{\end{equation}}
\def\ba{\begin{eqnarray}}
\def\ea{\end{eqnarray}}

\def\r{\mathbb{R}}

\def\<{\langle}
\def\>{\rangle}

\def\2{L^2}

\numberwithin{equation}{section}
\begin{document}
\title[Neumann controllability of the KdV equation]{ Neumann boundary controllability of the Korteweg-de Vries equation on a bounded domain}
\author[Caicedo]{Miguel Caicedo}
\address{Department of Mathematical Sciences, University of Cincinnati,
2815 Commons Way -French Hall West
Cincinnati, Ohio - 45221 United States}
\email{caicedms@mail.uc.edu}
\author[Capistrano--Filho]{Roberto  A. Capistrano--Filho}
\address{Department of Mathematical Sciences, University of Cincinnati,
2815 Commons Way -French Hall West
Cincinnati, Ohio - 45221 United States}
\email{capistranofilho@gmail.com}
\author[Zhang]{Bing-Yu Zhang}
\address{Department of Mathematical Sciences,
         University of Cincinnati,
          Cincinnati , Ohio 45221-0025 and Yangtze Center of mathematics, Sichuan University, Chengdu, China}
\email{zhangb@ucmail.uc.edu}
\date{}

\begin{abstract}
In this paper we study boundary controllability of the Korteweg-de Vries (KdV) equation posed on a finite domain $(0,L)$  with the Neumann boundary conditions:
\begin{equation}\label{x-1}
\left\{
\begin{array}
[c]{lll}%
u_t+u_x+uu_x+u_{xxx}=0 &  & \text{ in } (0,L)\times(0,T),\\
u_{xx}(0,t)=0,\text{ }u_x(L,t)=h(t), \text{ }u_{xx}(L,t)=0 &  & \text{ in }(0,T),\\
u(x,0)=u_0(x) & & \text{ in }(0,L).
\end{array}
\right.
\end{equation}
We show that the associated linearized system
\begin{equation}\label{x-2}
\left\{
\begin{array}
[c]{lll}%
u_t+(1+\beta) u_x+u_{xxx}=0 &  & \text{ in } (0,L)\times(0,T),\\
u_{xx}(0,t)=0,\text{ }u_x(L,t)=h(t), \text{ }u_{xx}(L,t)=0 &  & \text{ in }(0,T),\\
u(x,0)=u_0(x) & & \text{ in }(0,L),
\end{array}
\right.
\end{equation}
is exactly controllable if and only if  $\beta \ne -1$  and the length $L$ of the spatial domain $(0,L)$ does not belong to  set
\begin{equation*}
\mathcal{R}_{\beta} :=\left\{  \frac{2\pi}{\sqrt{3(1+\beta)}}\sqrt{k^{2}+kl+l^{2}}
\,:k,\,l\,\in\mathbb{N}^{\ast}\right\}\cup\left\{\frac{k\pi}{\sqrt{1+\beta}}:k\in\mathbb{N}^{\ast}\right\}.
\end{equation*}
Then the nonlinear system (\ref{x-1}) is  shown to be locally exactly controllable around a constant steady state $\beta$ if the associated  linear system is exactly controllable.
\end{abstract}
\maketitle

\section{Introduction\label{Sec0}}
In this paper we study a class of distributed parameter control system described by the Korteweg-de Vries (KdV) equation posed on a bounded interval $(0,L) $ with  the Neumann boundary conditions:
\begin{equation}
\left\{
\begin{array}
[c]{lll}%
u_t+u_x+uu_x+u_{xxx}=0 &  & \text{ in } (0,L)\times(0,T),\\
u_{xx}(0,t)=0,\text{ }u_x(L,t)=h(t),\text{ }u_{xx}(L,t)=0 &  & \text{ in }(0,T),\\
u(x,0)=u_0(x) & & \text{ in }(0,L),
\end{array}
\right.  \label{1}
\end{equation}
where the boundary value function $h=h(t)$  will be considered as a  control input. We are mainly concerned with its exact  control problem:
\vglue 0.2 cm
 {\em  Given $T>0$ and $u_0,u_T\in L^2(0,L)$, can one find  an appropriate control input $h$  such that the corresponding solution $u$ of \eqref{1} satisfies
$$u(x,0)=u_0(x), \text{  } \text{  } \text{  } \text{  }u(x,T)=u_T(x)?$$}

The study of  control and stabilization of the  KdV equation  begun with the works  of Russell \cite{russell3}, Zhang \cite{zhang4},  Russell and Zhang \cite{russell2,Russell1}   in which they studied   internal control of the
 KdV equation posed on a finite domain  $(0,L)$  with periodic boundary conditions.  Aided by then newly discovered Bourgain smoothing properties  \cite{bourgain1,bourgain2}
 they showed that the  internal control system is locally exactly controllable and exponentially stabilizable\footnote{The system has been shown  recently to be globally
 exponentially stabilizable and large time  exactly controllable by Laurent, Rosier and Zhang \cite{crz}.}.  Since then, control and stabilization  of the KdV equation have been
  intensively studied (see \cite{cerpa,cerpatut,cerpa1,coron,GG,GG1,Rosier,Rosier2,RZsurvey,zhang2} and references therein).
  In particular, Rosier \cite{Rosier}  studied boundary control of the KdV equation posed on the finite domain $(0,L)$ with the Dirichlet boundary conditions:
\begin{equation}
\left\{
\begin{array}
[c]{lll}%
u_t+u_x+uu_x+u_{xxx}=0 &  & \text{ in } (0,L)\times(0,T),\\
u(0,t)=0,\text{ }u(L,t)=0,\text{ }u_x(L,t)=g(t) &  & \text{ in }(0,T),\\
u(x,0)=u_0(x) & & \text{ in }(0,L),
\end{array}
\right.  \label{2}
\end{equation}
where boundary value function $g(t)$ is considered as a control input. Rosier  considered  first the associated linear system
\begin{equation}
\left\{
\begin{array}
[c]{lll}%
u_t+u_x+u_{xxx}=0 &  & \text{ in } (0,L)\times(0,T),\\
u(0,t)=0,\text{ }u(L,t)=0,\text{ }u_x(L,t)=g(t) &  & \text{ in }(0,T),\\
u(x,0)=u_0(x) & & \text{ in }(0,L)
\end{array}
\right.  \label{2a}
\end{equation}
and discovered the so-called {\em critical length} phenomena; whether the system (\ref{2a}) is exactly controllable depends on the length $L$ of the spatial domain $(0,L)$.

\smallskip
 \noindent
 {\bf Theorem A} (Rosier \cite{Rosier})  {\em The linear system  \eqref{2a}  is exactly controllable in the space  $L^2(0,L)$ if and only if the length $L$  of the spatial domain $(0,L)$ does not belong to the set}
\begin{equation}
\mathcal{N}:=\left\{  \frac{2\pi}{\sqrt{3}}\sqrt{k^{2}+kl+l^{2}}
\,:k,\,l\,\in\mathbb{N}^{\ast}\right\}  . \label{critical}
\end{equation}

\smallskip
  The
controllability result  of the linear system was then extended to the nonlinear system when $L\notin\mathcal{N}$.
\vglue 0.2 cm
\noindent
{\bf Theorem B }(Rosier \cite{Rosier}): \textit{Let $T>0$ be given and assume $L\notin\mathcal{N}$. There exists $\delta>0$ such for any $u_0,u_T\in L^2(0,L)$ with $$||u_0||_{L^2(0,L)}+||u_T||_{L^2(0,L)}\leq\delta,$$ one can  find a  control input $g\in L^2(0,T)$ such that the  nonlinear system \eqref{2} admits a unique solution $$u\in C([0,T];L^2(0,L))\cap L^2(0,T;H^1(0,L))$$
satisfying $$u(x,0)=u_0(x), \text{  } \text{  } \text{  } \text{  }u(x,T)=u_T(x).$$}
\vglue 0.2 cm
In the case of $L\in\mathcal{N}$, Rosier proved  in \cite{Rosier} that the associated linear system \eqref{2a} is not controllable; there exists a finite-dimensional subspace of $L^2(0,L)$, denoted by $\mathcal{M}=\mathcal{M}(L)$, which is unreachable from $0$ for the linear system.  More precisely, for every nonzero state $\psi\in\mathcal{M}$, $g\in L^2(0,T)$  and $u\in C([0,T];L^2(0,L))\cap L^2(0,T;H^1(0,L))$ satisfying \eqref{2a} and $u(\cdot,0)=0$, one has $u(\cdot,T)\neq\psi$.
A spatial domain $(0,L)$ is called \textit{critical}  for the system (\ref{2a}) if its domain length $L\in\mathcal{N}$.

When the spacial domain $(0,L)$ is critical, one usually would not expect the corresponding nonlinear system \eqref{2} to be exactly controllable  as the linear system \eqref{2a} is not.
It  thus came as a surprise when Coron and Cr\'epeau showed  in \cite{coron} that the nonlinear system (\ref{2}) is  still locally exactly controllable  even though its spatial domain is critical with its length   $L=2k\pi$ and  $k\in\mathbb{N}^{*}$ satisfying
$$\nexists(m,n)\in\mathbb{N}^{*}\times\mathbb{N}^{*} \text{}\text{ with }\text{} m^2+mn+n^2=3k^2 \text{}\text{ and }\text{} m\neq n.$$
For those values of $L$, the  unreachable space $\mathcal{M}$ of the associated linear system  is an  one-dimensional linear space generated by the function $1-cos(x)$.
As for the other types of the critical domains,   the nonlinear system (\ref{2})  was shown later  by Cerpa \cite{cerpa}, and  Cerpa and Cr\'epeau in \cite{cerpa1} to be locally, large time exactly controllable.
\vglue 0.2 cm
\noindent
{\bf Theorem C }(Cr\'epeau and Cerpa \cite{cerpa,cerpa1}): \textit{Let  $L\in\mathcal{N}$  be given.  There exists  a $T_L >0$ such that for any $T>T_L$ there exists $\delta>0$ such for any $u_0,\ u_T\in L^2(0,L)$ with $$||u_0||_{L^2(0,L)}+||u_T||_{L^2(0,L)}\leq\delta,$$ there exists $g\in L^2(0,T)$ such that the system \eqref{2} admits a unique solution $$u\in C([0,T];L^2(0,L))\cap L^2(0,T;H^1(0,L))$$
satisfying $$u(x,0)=u_0(x), \text{  } \text{  } \text{  } \text{  }u(x,T)=u_T(x).$$}

\vglue 0.2 cm
In this paper, we  will investigate the boundary control system (\ref{1})  of  the KdV equation posed on the finite domain $(0,L)$ with the Neumann boundary conditions to see if it possesses  similar control properties as that possessed by  the  boundary control system (\ref{2}) .
 First we will study the following linearized system associated to (\ref{1}),
\begin{equation}
\left\{
\begin{array}
[c]{lll}%
u_t+ (1+\beta )u_x+u_{xxx}=0 &  & \text{ in } (0,L)\times(0,T),\\
u_{xx}(0,t)=0,\text{ }u_x(L,t)=h(t),\text{ }u_{xx}(L,t)=0 &  & \text{ in }(0,T),\\
u(x,0)=u_0(x) & & \text{ in }(0,L),
\end{array}
\right.  \label{linear}
\end{equation}
where $\beta $ a given real constant. For any $\beta \ne -1$,  we define
\begin{equation}
\mathcal{R}_{\beta} :=\left\{  \frac{2\pi}{\sqrt{3(1+\beta)}}\sqrt{k^{2}+kl+l^{2}}
\,:k,\,l\,\in\mathbb{N}^{\ast}\right\}\cup\left\{\frac{k\pi}{\sqrt{\beta +1}}:k\in\mathbb{N}^{\ast}\right\}.
\label{critical_new}
\end{equation}
The following theorem is one of the  main results in this  paper.
\begin{theorem} \label{th1.1} \quad

\begin{itemize}
\item[(i)] If $\beta \ne -1$, the linear system (\ref{linear}) is exactly controllable in the space $L^2 (0,L)$  if and only if the length L of the spatial domain $(0, L)$ does not belong to the set $\mathcal{R}_{\beta}$.
\item[(ii)] If $\beta =-1$, then the system (\ref{linear}) is not  exact controllable in the space $L^2 (0,L)$ for any $L>0$.
\end{itemize}
\end{theorem}

The next theorem addressing controllability of the  nonlinear system (\ref{1}) is  our another main result of the paper.
\begin{theorem}
Let $T>0$,  $\beta \ne -1$ and $L\notin\mathcal{R}_{\beta} $ be given.  There exists  a $\delta>0$ such that for any $u_0,u_T\in L^2(0,L)$ with $$||u_0-\beta ||_{L^2(0,L)}+||u_T-\beta ||_{L^2(0,L)}\leq\delta,$$ one can find  a control input $h\in L^2(0,T)$ such that the system \eqref{1} admits unique solution
$$u\in C([0,T];L^2(0,L))\cap L^2(0,T;H^1(0,L))$$
satisfying
$$u(x,0)=u_0(x), \text{  } \text{  } \text{  } \text{  }u(x,T)=u_T(x).$$
\label{th1.2}
\end{theorem}

The following remarks are now in order:

\begin{remark}

 In the case of $\beta=0$,  $\mathcal{N} $ is a proper subset  of $\mathcal {R}_{0}$.  The linear system (\ref{linear}) has  more critical length domains than
 that of the linear system (\ref{2a}). In the case of $\beta =-1$, every $L>0$ is critical for the system (\ref{linear}). By contrast,  if we  remove the term $u_x$ from the equation
 in (\ref{2a}),  every $L>0$ is  not critical   for the system (\ref{2a}).
 \end{remark}
 \begin{remark}
Every constant $\beta $ is a  steady  state  solution of the system (\ref{1}),  but is not for the system (\ref{2}).  The system (\ref{2}) is only known to be locally exact controllable
around the origin.  By contrast,  the system (\ref{1}) is locally exactly controllable around  any constant steady state  $\beta$ as  long as $L\notin  \mathcal{R}_{\beta}$.

\end{remark}
Theorem \ref{th1.1} will be proved using the same approach that Rosier used to establish Theorem A.  However, one will encounter some difficulties that demand special attention.
 The adjoint system of the linear system (\ref{linear})  is given by
\begin{equation}
\left\{
\begin{array}
[c]{lll}%
\psi_t+(1+\beta )\psi_x+\psi_{xxx}=0 &  &\text{ in } (0,L)\times(0,T),\\
(1+\beta )\psi(0,t)+\psi_{xx}(0,t)=0&  &\text{ in } (0,T),\\
(1+\beta )\psi(L,t)+\psi_{xx}(L,t)=0&  &\text{ in } (0,T),\\
\psi_x(0,t)=0&  &\text{ in } (0,T),\\
\psi(x,T)=\psi_T(x) & &\text{ in }(0,L).
\end{array}
\right.  \label{5}
\end{equation}
It is well known that the exact controllability of system \eqref{linear} is equivalent to the following observability inequality for the adjoint system \eqref{5}:
\begin{equation}
||\psi_T||_{L^2(0,L)}\leq C||\psi_x(L,\cdot)||_{L^2(0,T)}.
\label{6}
\end{equation}
However, the usual multiplier method and compactness arguments as those used in dealing with the  system \eqref{5} only lead to
\begin{equation}
||\psi_T||^2_{L^2(0,L)}\leq C_1||\psi_x(L,\cdot)||^2_{L^2(0,T)}+C_2||\psi(L,\cdot)||^2_{L^2(0,T)}.\label{7}
\end{equation}
One has to find a way to  remove the extra term present in \eqref{7}. For this, a technical lemma presented below, which reveals  some hidden regularities (or sharp trace regularities)
 for solutions of the adjoint system
\eqref{5}, is needed.

\begin{lemma}\label{sharp_int} (Hidden regularities)
For any $\psi_T\in L^2(0,L)$, the solution $$\psi\in C([0,T];L^2(0,L))\cap L^2(0,T;H^1(0,L))$$ of IBVP \eqref{5} possesses the following sharp trace properties
\begin{equation}
\sup_{x\in(0,L)}||\partial^r_x\psi(x,\cdot)||_{H^{\frac{1-r}{3}}(0,T)}\leq C_r||\psi_0||_{L^2(0,L)}
\label{9}
\end{equation}
for $r=0,1,2$.
\end{lemma}
The sharp Kato smoothing properties   of solutions of the Cauchy problem of the KdV  equation posed on the whole line $\mathbb{R}$ due to Kenig, Ponce and Vega \cite{KePoVe} will play an important role in the  proof of  Lemma \ref{sharp_int}.

\medskip
Following the work of  Rosier \cite{Rosier}, the boundary control system of the KdV equation posed on the finite interval $(0,L)$
with various control inputs has been intensively studied  (cf.  \cite{GG,GG1,Gui}  and see \cite{cerpatut, RZsurvey} for more  complete reviews)
\begin{equation}
\left\{
\begin{array}
[c]{lll}%
u_t+u_x+uu_x+u_{xxx}=0 &  & \text{ in } (0,L)\times(0,T),\\
u(0,t)=g_1 (t),\text{ }u(L,t)=g_2 (t),\text{ }u_x(L,t)=g_3(t) &  & \text{ in }(0,T),\\
u(x,0)=u_0(x) & & \text{ in }(0,L).
\end{array}
\right.  \label{m2}
\end{equation}
The system \eqref{m2} has been found to have an interesting property: it behaves like a parabolic system if control only applied on the left end of the spatial domain $(0,L)$
($g_2=g_3\equiv 0$): the system  is only null controllable;  but if control is allowed to apply on the right end  of the spatial domain $(0,L)$, the system
 behaves like a hyperbolic system which is exactly controllable. Moreover, the critical length phenomena occurs only in the case that just one control
 is applied to the right end of the spatial domain $(0,L)$.

 In this paper we will also show that the boundary control systems of the KdV equation posed on $(0,L)$ with Neumann boundary conditions,
 \begin{equation}
\left\{
\begin{array}
[c]{lll}%
u_t+u_x+uu_x+u_{xxx}=0 &  & \text{ in } (0,L)\times(0,T),\\
u_{xx}(0,t)=h_1(t),\text{ }u_x(L,t)=h_2(t),\text{ }u_{xx} (L,t)=h_3 (t)
 &  & \text{ in }(0,T),\\
u(x,0)=u_0(x) & & \text{ in }(0,L),
\end{array}
\right.  \label{m1}
\end{equation}
possess  similar properties.

\begin{theorem}
Let $T>0$ and $L>0$ be given. There exists $\delta>0$ such that for any $u_0,u_T\in L^2(0,L)$ with
$$
 ||u_0||_{L^2(0,L)}+||u_T||_{L^2(0,L)}\leq\delta,
$$
one can find  $h_1, h_2, $ and $h_3$ satisfy one of the conditions below
\begin{itemize}
\item[(i)]  $h_1\in H^{-\frac{1}{3}}(0,T)$ and $h_2\in L^2(0,T)$, $h_3=0$;
\item[(ii)]$h_2\in L^2 (0,T)$, $h_3\in H^{-\frac{1}{3}} (0,T)$ and $h_1=0$;
\item[(iii)] $ h_1, \ h_3 \in H^{-\frac{1}{3}}(0,T)$,  $h_2=0$,
\end{itemize}
 such that the system \eqref{m1} admits unique
solution
$$
 u\in C([0,T];L^2(0,L))\cap L^2(0,T;H^1(0,L))
$$
satisfying
$$u(x,0)=u_0(x), \text{  } \text{  } \text{  } \text{  }u(x,T)=u_T(x).$$
\label{th3_int}
\end{theorem}

If all three boundary control inputs are used,  then the system (\ref{m1}) has much stronger controllability; it is  locally exactly controllable
around any smooth solution of the KdV equation in the space $H^s (0,L)$ for any $s\geq 0$ and is large time globally exactly controllable in  the space $H^s (0,L)$ for any $s\geq 0$.
\begin{theorem}
Let $ T > 0 $, $s\geq 0$  and $ L > 0 $ be given. Assume that $ y \in C^{\infty}(\mathbb{R},H^{\infty}(\mathbb{R}))$ satisfies
$$
 y_t+y_x+yy_x+y_{xxx}=0 \text{ }\text{ }(x,t)\in\mathbb{R}\times\mathbb{R}.
$$
Then, there exists $\delta>0$ such that for any $y_0,y_T\in H^s (0,L)$ with
$$
 ||u_0-y(\cdot,0)||_{H^s(0,L)}+||u_T-y(\cdot,T)||_{H^s (0,L)}\leq\delta,
$$
one can find
$$
 h_1\in H^{\frac{s-1}{3}}(0,T),\text{ } h_2\in H^{\frac{s}{3}}(0,T),\text{ }  h_3\in H^{\frac{s-1}{3}}(0,T)
$$
such that system \eqref{m1} admits a unique solution
$$
 u\in C([0,T];H^s (0,L))\cap L^2(0,T;H^{s+1}(0,L))
$$
satisfying
$$
 u(x,0)=u_0(x), \text{  } \text{  } \text{  } \text{  }u(x,T)=u_T(x).
$$
\label{th6_int}
\end{theorem}
 \begin{theorem} Let $L>0$ and $\gamma >0$ be given. There exists a $T>0$ such that for any $u_0, \ u_T\in L^2 (0,L)$ satisfying
 \[ \|u_0\|_{L^2 (0,L) } +\|u_T\|_{L^2 (0,L)} \leq \gamma, \]
 one can find
$$
 h_1\in H^{-\frac{1}{3}}(0,T),\text{ } h_2\in L^2 (0,T),\text{ }  h_3\in H^{-\frac{1}{3}}(0,T)
$$
such that system \eqref{m1} admits a unique solution
$$
 u\in C([0,T];L^2 (0,L))\cap L^2(0,T;H^{1}(0,L))
$$
satisfying
$$
 u(x,0)=u_0(x), \text{  } \text{  } \text{  } \text{  }u(x,T)=u_T(x).
$$
\label{th6_int_z}
 \end{theorem}
Finally, if we consider the system  with only control acting on the left end of the spatial domain $(0,L)$,
\begin{equation}
\left\{\begin{array}[c]{lll}%
        u_t+u_x+u_{xxx}+uu_x=0 &  & \text{ in } (0,L)\times(0,T),\\
        u_{xx}(0,t)=h_1(t),\text{ }u_x(L,t)=0,\text{ }u_{xx}(L,t)=0 &  & \text{ in }(0,T),\\
        u(x,0)=u_0(x)& & \text{ in }(0,L),
       \end{array}
\right.  \label{1abc}
\end{equation}
 we have the following null controllability result.
\begin{theorem}[Null Controllability]
Let $ T > 0 $ be fixed. For $\overline{u}_0\in L^2(0,L)$,  let 
$$
 \overline{u}\in C([0,T];L^2(0,L))\cap L^2(0,T;H^1(0,L))
$$
be the solution of the following system
\begin{equation}
\left\{\begin{array}[c]{lll}%
        \overline{u}_t+\overline{u}_x+\overline{u}_{xxx}+\overline{uu}_{x}=0&  & \text{ in } (0,L)\times(0,T),\\
        \overline{u}_{xx}(0,t)=0,\text{ }\overline{u}_x(L,t)=0,\text{ }\overline{u}_{xx}(L,t)=0&  & \text{ in }(0,T),\\
        \overline{u}(x,0)=\overline{u}_0(x)& & \text{ in }(0,L).
       \end{array}
\right.  \label{1ab}
\end{equation}
Then, there exists $\delta>0$ such that for any $u_0\in L^2(0,L)$ satisfying
$$
 ||u_0-\overline{u}_0||_{L^2(0,L)} < \delta,
$$
there exists $h_1(t)\in L^2(0,T)$ such that the solution $u(x,t)$ of the system \eqref{1abc} belongs to the space $$C([0,T];L^2(0,L))\cap L^2(0,T;H^1(0,L))$$ and satisfies
$$u(x,T)=\overline{u}(x,T) \text{  } \text{ in  }(0,L).$$
\label{null}
\end{theorem}

The paper is organized as follows.

\medskip
----  In Section \ref{Sec2},  we  study the  non-homogeneous boundary value problem  of the KdV equation on the finite interval $(0,L)$,
 \begin{equation}
\left\{
\begin{array}
[c]{lll}%
u_t+u_x+uu_x+u_{xxx}=0 &  & \text{ in } (0,L)\times(0,T),\\
u_{xx}(0,t)=h_1(t),\text{ }u_x(L,t)=h_2(t),\text{ }u_{xx} (L,t)=h_3 (t)
 &  & \text{ in }(0,T),\\
u(x,0)=u_0(x) & & \text{ in }(0,L),
\end{array}
\right.  \label{m3}
\end{equation}
for its well-posedness  in the space $L^2 (0,L)$. We will show that  the system (\ref{m3}) is locally well-posed in the space $L^2 (0,L)$: for any $u_0\in L^2 (0,L)$,
\[ \mbox{$h_1\in H^{-\frac13} (\mathbb{R}^+), \ h_2 \in L^2 (\mathbb{R}^+), $ and $h_3\in H^{-\frac13} (\mathbb{R}^+),$ }\]
 there exists a $T>0$ such that  (\ref{m3}) admits a unique solution $$u\in C([0,T]; L^2 (0,L))\cap L^2 (0,T; H^1 (0,L)).$$
 Various linear estimates including hidden regularities  will be  presented for solutions of the linear system associated to (\ref{m3}), which will play important roles in studying controllability of the system.

 \medskip
---- In Section \ref{Sec3},   the boundary control system (\ref{1}) will be investigated for its controllability.  We investigate  first the linearized system (\ref{linear}) and its corresponding adjoint system \eqref{5} for their controllability and observability. In particular,  the hidden regularities for solutions of the adjoint system \eqref{5}
will be presented and then be used to prove Theorem \ref{th1.1} and Theorem \ref{th1.2}.

 \medskip
---- The sketch of proofs of Theorem \ref{th3_int} and Theorem \ref{null} will be presented in Section \ref{Sec4} together with the proofs of Theorem \ref{th6_int} and Theorem \ref{th6_int_z}.

 \medskip
 We end our introduction with a few more comments.  Having shown the nonlinear  system (\ref{1})  is locally exactly controllable (Theorem \ref{th1.2}) if the the length  of the spatial domain is not critical,
one naturally would  like to show the system (\ref{1}) is still locally exactly controllable  when the length of its spatial domain is critical as in the  case of system (\ref{2})  (see Theorem C).
We believe that  the same approach developed  in \cite{cerpa,coron02,coron}  to prove  Theorem C for the system (\ref{2})  can be adapted to obtain  similar results for system (\ref{1}).
However,
the Neumann boundary conditions present  some extra difficulties.  In particular,  the adjoint linear system associated to \eqref{1} is different from   the adjoint  linear system associated to \eqref{2}.
The unobservable solutions of the adjoint system associated to \eqref{1}  are not  the solutions of  the forward linear system  associated to \eqref{1}.   When using  the power series expansion method proposed in \cite{coron02},  there are more terms appeared  that  demand special handling.   We plan to continue to study system (\ref{1})  with the critical length for its controllability and present our results in the forthcoming paper.

\section{Well-posedness\label{Sec2}}
In this section, we  will show the initial-boundary value problem
\begin{equation}
\left\{
\begin{array}
[c]{lll}%
u_t+u_x+uu_x+u_{xxx}=0 &  & \text{ in } (0,L)\times(0,T),\\
u_{xx}(0,t)=h_1(t),\text{ }u_x(L,t)=h_2(t),\text{ }u_{xx}(L,t)=h_3(t) &  & \text{ in }(0,T),\\
u(x,0)=u_0(x) & & \text{ in }(0,L),
\end{array}
\right.  \label{1-a}
\end{equation}
is  locally well-posed in the space $L^2  (0,L)$ with $u_0(x) \in L^2 (0,L)$ and
\[ h_1, \ h_3 \in H^{-\frac13}(0,T), \qquad h_2\in L^2 (0,T)  .\]

We begin by considering the  following linear non-homogeneous  boundary value  problem,
\begin{equation}
\left\{
\begin{array}
[c]{lll}%
w_t+w_{xxx}=0,  \quad w(x,0)=0  &  & x\in(0,L), t>0,\\
w_{xx}(0,t)=h_1(t),\text{ }w_x(L,t)=h_2(t),\text{ }w_{xx}(L,t)=h_3(t)&  & t>0. \end{array}
\right.  \label{w1b}
\end{equation}
First, we  derive  an explicit solution formula for  its solution.
Applying the Laplace transform with respect to $t$, \eqref{w1b} is converted to
\begin{equation}
\left\{
\begin{array}
[c]{l}%
s\hat{w}+\hat{w}_{xxx}=0,\\
\hat{w}_{xx}(0,s)=\hat{h}_1(s),\text{ }\hat{w}_x(L,s)=\hat{h}_2(s),\text{ }\hat{w}_{xx}(L,s)=\hat{h}_3(s),
\end{array}
\right.  \label{w2}
\end{equation}
where
$$\hat{w}(x,s)=\int_0^{+\infty}e^{-st}w(x,t)dt$$
and
$$\hat{h}_j(s)=\int_0^{+\infty}e^{-st}h(t)dt, \ \ j=1,2,3.$$
The solution $\hat{w}(x,s)$ can be written in the form
$$\hat{w}(x,s)=\sum_{j=1}^3 c_j(s)e^{\lambda_j(s)x},$$
where $\lambda_j(s)$, $j=1,2,3$ are the solutions of the characteristic equation
$$s+\lambda^3=0$$
and $c_j(s)$, $j=1,2,3$, solve the linear system
\begin{equation}
\underbrace{\begin{pmatrix}
\lambda_1^2 & \lambda_2^2 & \lambda_3^2  \\
\lambda_1e^{\lambda_1L} & \lambda_2e^{\lambda_2L} & \lambda_3e^{\lambda_3L}  \\
\lambda_1^2e^{\lambda_1L} & \lambda_2^2e^{\lambda_2L} & \lambda_3^2e^{\lambda_3L}
\end{pmatrix}}_A
\begin{pmatrix}
c_1  \\
c_2  \\
c_3
\end{pmatrix}=
\underbrace{\begin{pmatrix}
\hat{h}_1 \\
\hat{h}_2  \\
\hat{h}_3
\end{pmatrix}.}_{\vec{h}}
\label{w3}
\end{equation}
By Cramer's rule,
$$c_j=\frac{\Delta_j(s)}{\Delta(s)}, \ \ j=1,2,3,$$
with $\Delta$ the determinant of $A$ and $\Delta_j$ the determinant of the matrix $A$ with the $j$th-column replaced by $\vec{h}$. The solution $w(x,t)$ for \eqref{w1b} can be written in the form
\begin{equation}\label{rep}
w(x,t)=\sum_{m=1}^3w_m(x,t),
\end{equation}
where $w_m(x,t)$ solves \eqref{w1b} with $h_j\equiv0$ when $j\neq m$, $j,m=1,2,3$. Using the inverse Laplace transform yields
$$w(x,t)=\frac{1}{2\pi i}\int^{r+i\infty}_{r-i\infty}e^{st}\hat{w}(x,s)ds=\sum_{j=1}^{3}\frac{1}{2\pi i}\int^{r+i\infty}_{r-i\infty}\frac{\Delta_{j}(s)}{\Delta(s)}e^{\lambda_j(s)x}ds,$$
for $r>0$. Combining this with \eqref{rep} we can write the values of $w_m$ as follows for $m=1,2,3$,
  $$w_m(x,t)=\sum_{j=1}^{3}\frac{1}{2\pi i}\int^{r+i\infty}_{r-i\infty}\frac{\Delta_{j,m}(s)}{\Delta(s)}e^{\lambda_j(s)x}\hat{h}_m(s)ds\equiv[W_{m,j}(t)h_m](x).$$ In the last two formulas, the right-hand side are continuous with respect to $r$ for $r\geq0$. As the left-hand sides do not depend on $r$, we can take $r=0$ in these formulas. Moreover,
$$w_{j,m}(x,t)=w^+_{j,m}(x,t)+w^-_{j.m}(x,t)$$
where
$$w^+_{j,m}(x,t)=\frac{1}{2\pi i}\int_0^{+i\infty}e^{st}\frac{\Delta_{j,m}(s)}{\Delta(s)}\hat{h}_m(s)e^{\lambda_j(s)x}ds$$
and
$$w^-_{j,m}(x,t)=\frac{1}{2\pi i}\int_{-i\infty}^0e^{st}\frac{\Delta_{j,m}(s)}{\Delta(s)}\hat{h}_m(s)e^{\lambda_j(s)x}ds,$$
for $j,m=1,2,3$. Here $\Delta_{j,m}(s)$ is obtained from $\Delta_j(s)$ by letting $\hat{h}_m(s)=1$ and $\hat{h}_k(s)=0$ for $k\neq m$, $k,m=1,2,3$. Making the substitution $s=i\rho^3$ with $\rho\geq0$ in the characteristic equation
$$s+\lambda^3=0.$$
The three roots are given in terms of $\rho$ by
\begin{equation}
\lambda_1(\rho)=i\rho,\ \ \ \lambda_2(\rho)=-i\rho\left(\frac{1+i\sqrt{3}}{2}\right), \ \ \ \lambda_3(\rho)=-i\rho\left(\frac{1-i\sqrt{3}}{2}\right),
\label{w4}
\end{equation}
thus $w^+_{j,m}$ has the following form
$$w^+_{j,m}(x,t)=\frac{1}{2\pi i}\int_0^{+\infty}e^{i\rho^3t}\frac{\Delta^+_{j,m}(\rho)}{\Delta^+(\rho)}\hat{h}^+_m(\rho)e^{\lambda^+_j(\rho)x}3i\rho^2d\rho$$
and
$$w^-_{j,m}(x,t)=\overline{w^+_{j,m}(x,t)},$$
where $\hat{h}^+_m(\rho)=\hat{h}_m(i\rho^3)$, $\Delta^+(\rho)=\Delta(i\rho^3)$, $\Delta^+_{j,m}(\rho)=\Delta_{j,m}(i\rho^3)$ and $\lambda^+_j(\rho)=\lambda_j(i\rho^3)$.

\medskip
Then the solution of the IBVP \eqref{w1b} has the following representation.

\begin{lemma}\label{lem1a}
Given  $\vec{h}=(h_1,h_2,h_3)$, the  solution $w$ of the IBVP \eqref{w1b} can be written in the form
\[
w(x,t)=[W_{bdr}\vec{h}](x,t):= \sum_{j,m=1}^3[W_{j,m}h_m](x,t).
\]
\end{lemma}


Let $\vec{h}:=(h_1,h_2,h_3)\in\mathcal{H}_T$ with $$\mathcal{H}_T=H^{-\frac{1}{3}}(0,T)\times L^2(0,T)\times H^{-\frac{1}{3}}(0,T)$$ and $$\mathcal{Z}_T:=C([0,T];L^2(0,L))\cap L^2(0,T;H^1(0,L)).$$ The following lemma holds, for solution of the system \eqref{w1b}.

\begin{proposition}\label{l1}
Let $T>0$ be given. There exists a constant $C>0$ such that for any $\vec{h}\in\mathcal{H}_T$ the system \eqref{w1b} admits a unique solution $w\in\mathcal{Z}_T$. Moreover, there exists a constant $C>0$ such that
$$||w||_{\mathcal{Z}_T}+\sum_{j=0}^2||\partial^j_xw||_{L^{\infty}(0,L;H^{\frac{1-j}{3}}(0,T))}\leq C||\vec{h}||_{\mathcal{H}_T}.$$
\end{proposition}
\begin{proof}
As we stated above, the solution $w$ can be written as
$$w(x,t)=w_1(x,t)+w_2(x,t)+w_3(x,t).$$
We just  prove  Proposition  \ref{l1} for $w_1$.  The proof for $w_2 $ and $w_3$ are similar.  Some straightforward calculations show that the asymptotic behavior of the ratios $\frac{\Delta^+_{j,m}(\rho)}{\Delta^+(\rho)}$ as $\rho\to+\infty$ are:
\begin{equation*}
\begin{array}{||l|l|l||}
\hline
\frac{\Delta^+_{1,1}(\rho)}{\Delta^+(\rho)} \sim \rho^{-2}e^{-\frac{\sqrt3}{2}\rho L}& \frac{\Delta^+_{2,1}(\rho)}{\Delta^+(\rho)} \sim \rho^{-2}e^{-\frac{\sqrt3}{2}\rho L} & \frac{\Delta^+_{3,1}(\rho)}{\Delta^+(\rho)} \sim \rho^{-2}e^{-\frac{\sqrt3}{2}\rho L}\\
\hline
\frac{\Delta^+_{1,2}(\rho)}{\Delta^+(\rho)} \sim \rho^{-1} & \frac{\Delta^+_{2,2}(\rho)}{\Delta^+(\rho)} \sim \rho^{-1} & \frac{\Delta^+_{3,2}(\rho)}{\Delta^+(\rho)} \sim \rho^{-1}\\
\hline
\frac{\Delta^+_{1,3}(\rho)}{\Delta^+(\rho)} \sim \rho^{-2} & \frac{\Delta^+_{2,3}(\rho)}{\Delta^+(\rho)} \sim \rho^{-2}e^{-\frac{\sqrt3}{2}\rho L} & \frac{\Delta^+_{3,3}(\rho)}{\Delta^+(\rho)} \sim \rho^{-2} \\
\hline
\end{array}
\label{behavior}
\end{equation*}
Since
$$w_1(x,t)=\frac{3}{\pi}\sum_{j=1}^3\int_0^{+\infty}e^{i\rho^3t}e^{\lambda_j^+(\rho)x}\frac{\Delta^+_{j,1}(\rho)}{\Delta^+(\rho)}\hat{h}^+_1(\rho)\rho^2d\rho,$$
we have
\begin{align*}
\sup_{t\in(0,T)}||w_1(\cdot,t)||^2_{L^2(0,L)} &  \leq C\int_0^{\infty}\mu^{-2/3}|\hat{h}^+_1(i\mu)|^2d\mu\\
&  \leq C||h_1||^2_{H^{-\frac{1}{3}}(\mathbb{R}^+)}\\
&  \leq C||\vec{h}||_{\mathcal{H}_T}.
\end{align*}
Furthermore, for $l=0,1,2,$ let us to consider $\theta(\mu)$ the real solution of $\mu=\rho^3$, $\rho\geq0$, thus
\begin{align*}
\partial^l_xw_1(x,t)&=\frac{3}{\pi}\sum_{j=1}^3\int_0^{+\infty}\left(\lambda^+_j(\rho)\right)^le^{i\rho^3t}e^{\lambda_j^+(\rho)x}\frac{\Delta^+_{j,1}(\rho)}{\Delta^+(\rho)}\hat{h}^+_1(\rho)\rho^2d\rho\\
& = \frac{3}{\pi}\sum_{j=1}^3\int_0^{+\infty}\left(\lambda^+_j(\theta(\mu))\right)^le^{i\rho^3t}e^{\lambda_j^+(\theta(\mu))x}\frac{\Delta^+_{j,1}(\theta(\mu))}{\Delta^+(\theta(\mu))}\hat{h}^+_1(i\mu)d\mu.
\end{align*}
Applying Plancherel theorem, in time $t$, yields that, for all $x\in(0,L)$
\begin{align*}
||\partial^{l}_xw_1(x,\cdot)||^2_{H^{\frac{1-l}{3}}(0,T)}&\leq C\sum_{j=1}^3\int_0^{+\infty}\mu^{\frac{2(1-l)}{3}}\left|(\lambda^+_j(\theta(\mu))^{\rho+1})e^{\lambda_j^+(\theta(\mu))x}\frac{\Delta^+_{j,1}(\theta(\mu))}{\Delta^+(\theta(\mu))}\hat{h}^+_1(i\mu)\right|^2d\mu\\
& \leq C\int_0^{+\infty}\mu^{-\frac{2}{3}}\left| \hat{h}_1(i\mu)\right|^2d\mu\\
& \leq C||h_1||^2_{H^{-\frac{1}{3}}(0,T)}\\
& \leq C||\vec{h}||^2_{\mathcal{H}_T},
\end{align*}
for $l=0,1,2$. Consequently
$$ \sup_{x\in(0,L)}||\partial^{l}_xw_1(x,\cdot)||_{H^{\frac{(1-l)}{3}}(0,T)}\leq C||\vec{h}||_{\mathcal{H}_T}, \text{ }l=0,1, 2$$
which ends the proof of Proposition \ref{l1} for $w_1$. 
\end{proof}

Next we turn to  consider the following initial boundary-value problem:
\begin{equation}
\left\{
\begin{array}
[c]{lll}%
v_t+v_{xxx}=f &  & x\in(0,L), t>0,\\
v_{xx}(0,t)=0,\text{ }v_x(L,t)=0,\text{ }v_{xx}(L,t)=0&  & t>0,\\
v(x,0)=\phi(x) & & x\in(0,L).
\end{array}
\right.  \label{w1}
\end{equation}
By semigroup theory,  for any $\phi \in L^2 (0,L)$ and $f\in L^1 (0,T; L^2 (0,L))$, it possess a unique  mild solution $u\in C([0,T]; L^2 (0,L))$ which can be written as
\[ u(x,t)= W_0 (t) \phi + \int ^t_0 W_0 (t-\tau ) f(\tau ) d\tau .\]
Here  $\{W_0(t)\} _{t\geq 0}$ is the $C_0$-semigroup in the space $L^2(0,L)$ generated by linear operator
$$Ag=-g''$$
with domain
$$\mathcal{D}(A)=\{g\in H^3(0,L):g''(0)=g'(L)=g''(L)=0\}.$$ In order to show that the solution $u$  of (\ref{w1}) also possesses the Kato smoothing property
\[ u\in L^2 (0,T; H^1 (0,L))\]
and hidden regularity (sharp Kato smoothing property),
\[ \partial ^k_x u\in L^{\infty}_x (0,L; H^{\frac{1-k}{3}} (0,T)), \ k=0,1,2, \]
we rewrite $u$ in terms of boundary integral operator $W_{bdr} (t)$ and  the solution of the following initial value problem (IVP) of  the  linear KdV equation posed on the whole line $\mathbb{R}$,
\begin{equation}\label{RP}
        \begin{cases}
            v_t+v_{xxx}=g(x,t) \qquad x \in \r,  \ t\in \r^+,\\
            v(x,0)=\psi (x). \
        \end{cases}
\end{equation}
Recall that  the solution$v(x,t) $ can be written as
\[ v(x,t) = W_{\r} (t) \psi   +\int ^t_0 W_{\r} (t-\tau) g(\tau ) d\tau, \]
where $\{W_{\r} (t)\} _{t\in \r}  $ is the $C_0$ group in the space $L^2 (0,L)$ generated by the operator $Kg=-g''' $ with domain $\mathcal{D} (K)=H^3 (\r)$.  The following results are well--known for solutions of (\ref{RP}) (see e.g. \cite{KePoVe}).
\begin{proposition}\label{l2}
 For any $\psi \in L^2 (\r)$ and $g\in L^1 (\r; L^2 (\r))$,   (\ref{RP}) admits a  unique mild solution $v\in C(\r; L^2 (\r))$ satisfying
 \[ \| v(\cdot, t)\|_{L^2 (\r)} = \|\psi \|_{L^2 (\r)} \quad \text{for any } \ t\in \r .\]
 Moreover, the the solution $v$ possesses the local Kato smoothing property $v\in L^2 (\r; H^1(-L,L)) $ for any $L>0$ with
 \[ \| v\|_{L^2 (\r, H^1 (-L,L))} \leq C_L\left (  \|\psi\|_{L^2 (\r)} + \|g\|_{L^1 (\r; L^2 (\r))}\right ) \]
 and the sharp Kato smoothing properties $\partial ^k_x v \in L^{\infty}_x(\r; H^{(k-1)/3}_t (\r))$  with
 \[ \|\partial ^k_xv\|_{L^{\infty}_x (\r; H^{(k-1)/3} (\r))} \leq C_k\left (  \|\psi\|_{L^2 (\r)} + \|g\|_{L^1 (\r; L^2 (\r))}\right ) \]
 for $k=0,1,2.$

\end{proposition}
For $\phi \in L^2 (0,L)$ and $f\in L^1(0,T; L^2 (0,L)$, let
\[ \tilde{\phi} (x) = \left \{ \begin{array}{ll} \phi (x) & \ if \ x\in (0,L), \\
0&\ if \ x\notin(0,L) \end{array} \right.  \]
and
\[ \tilde{f} (x,t)= \left \{ \begin{array}{ll} f(x,t)& \ if \ x\in (0,L) \times (0,T), \\
0&\ if \ x\notin(0,L)\times (0,T). \end{array} \right.  \]
We have  $\tilde{\phi}\in L^2 (\r)$,  $\tilde{f}\in L^1 (\r; L^2 (\r))$
and
\[ \|\tilde{\phi}\|_{L^2 (\r)}=\|\phi\|_{L^2 (0,L)}, \quad \|\tilde{f}\|_{L^1 (\r; L^2 (\r))}=\|f\|_{L^1 (0,T; L^2 (0,L))}.\]
Let
\[ v(x,t)=W_{\r} (t)\tilde{\phi} +\int ^t_0 W_{\r} (t-\tau )\tilde{f}(\tau) d\tau \]
and
\[ q_1 (t)=v_{xx}(0,t), \quad q_2 (t)=v_x(L,t), \quad q_3 (t) =v_{xx} (L,t), \quad \vec{q} (t)=(q_1 (t), q_2 (t), q_3 (t)) .\]
Then $v(x,t) $ solves  \eqref{RP} with $\psi$ and $g$ replaced by $\tilde{\phi} $ and $\tilde{f}$, respectively, and
\[ u(x,t) =W_{\r} (t)\tilde{\phi} +\int ^t_0 W_{\r} (t-\tau )\tilde{f}(\tau) d\tau- W_{bdr} (t)\vec{q} \]
solves the IBVP (\ref{w1}).
 The following proposition then follows from Propositions \ref{l1} and \ref{l2}.
 \begin{proposition}\label{l3}
 Let $T>0$ be given.  For any $\phi \in L^2 (0,L)$ and $f\in L^1 (0,T; L^2 (0,L))$,  the IBVP (\ref{w1}) admits   a unique  mild solution $v\in C(\r; L^2 (0,L))$ satisfying
 \[ \| v(\cdot, t)\|_{L^2 (0,L)} \leq C\|\psi \|_{L^2 (0,L)} \quad \text{for any } \ t\in \r .\]
 Moreover, the the solution $v$ possesses the local Kato smoothing property $v\in L^2 (0,T; H^1(0,L)) $   with
 \[ \| v\|_{L^2 (0,T; H^1 (0,L))} \leq C\left (  \|\phi\|_{L^2 (0,L)} + \|f\|_{L^1 (0,T; L^2 (0,L))}\right ) \]
 and the sharp Kato smoothing properties $\partial ^k_x v \in L^{\infty}_x(0,L; H^{(k-1)/3} (0,T))$  with
 \[ \|\partial ^k_xv\|_{L^{\infty}_x (0,L; H^{(k-1)/3} (0,T))} \leq C_k\left (  \|\phi\|_{L^2 (0,L)} + \|f\|_{L^1 (0,T; L^2 (0,L))}\right ) \]
 for $k=0,1,2.$

 \end{proposition}

Combining Proposition \ref{l2} and Proposition \ref{l3} together leads the result for the following IBVP:
\begin{equation}
\left\{
\begin{array}
[c]{lll}%
v_t+v_{xxx}=f&  & \text{ in } (0,L)\times(0,T),\\
v_{xx}(0,t)=h_1,\text{ }v_x(L,t)=h_2,\text{ }v_{xx}(L,t)=h_3&  & \text{ in }(0,T),\\
v(x,0)=v_0(x)& & \text{ in }(0,L).
\end{array}
\right.  \label{w6}
\end{equation}
\begin{proposition}\label{p3}
Let $T>0$ be given, for any $v_0\in L^2(0,L)$, $f\in L^1(0,T;L^2(0,L))$ and
$$\vec{h}:=(h_1,h_2,h_3)\in\mathcal{H}_T=H^{-\frac{1}{3}}(0,T)\times L^2(0,T)\times H^{-\frac{1}{3}}(0,T),$$
the IBVP \eqref{w6} admits a unique solution
$$v\in\mathcal{Z}_T:=C([0,T];L^2(0,L))\cap L^2(0,T;H^1(0,L))$$
with
\[ \partial ^k_x v\in L^{\infty}_x (0,L; H^{\frac{1-k}{3}} (0,T)\quad \text{for} \ k=0,1,2 .\]
Moreover, there exists $C>0$ such that
$$||v||_{\mathcal{Z}_T} +\sum ^{2}_{k=0} \| v\|_{L^{\infty}_x (0,L; H^{\frac{1-k}{3}} (0,T)} \leq C\left(||v_0||_{L^2(0,L)}+||\vec{h}||_{\mathcal{H}_T}+||f||_{L^1(0,T;L^2(0,L))}\right).$$
\end{proposition}

Next proposition states similar hidden (or sharp trace) regularities for the linear system
\begin{equation}
\left\{
\begin{array}
[c]{lll}%
y_t+y_x+y_{xxx}=f &  &x\in(0,L), t>0\\
y(0,t)=g_1(t),\text{ }y_x(L,t)=g_2(t),\text{ }y_{xx}(L,t)=g_3(t) &  &t>0,\\
y(x,0)=y_0(x)& &x\in(0,L),
\end{array}
\right.  \label{w8}
\end{equation}
associated to \eqref{2}.
\begin{proposition}\label{p5}
Let $T>0$ be given, for any $y_0\in L^2(0,L)$, $f\in L^1(0,T;L^2(0,L))$ and
$$\vec{g}:=(g_1,g_2,g_3)\in\mathcal{G}_T=H^{\frac{1}{3}}(0,T)\times L^2(0,T)\times H^{-\frac{1}{3}}(0,T),$$
the IBVP \eqref{w8} admits a unique solution $y\in\mathcal{Z}_T$. Moreover, there exists $C>0$ such that
$$||y||_{\mathcal{Z}_T}\leq C\left(||y_0||_{L^2(0,L)}+||\vec{g}||_{\mathcal{G}_T}+||f||_{L^1(0,T;L^2(0,L))}\right).$$
In addition, the solution $y$ possesses the following sharp trace estimates
\begin{equation}
\sup_{x\in(0,L)}||\partial^r_xy(x,\cdot)||_{H^{\frac{1-r}{3}}(0,T)}\leq C_r\left(||y_0||_{L^2(0,L)}+||\vec{g}||_{\mathcal{G}_T}+||f||_{L^1(0,T;L^2(0,L))} \right),
\label{w9}
\end{equation}
for $r=0,1,2$.
\end{proposition}
The proof of Proposition \ref{p5} can be found in \cite{BonaSunZhang,KrRiZha}.
\begin{remark}
Systems \eqref{w6} and \eqref{w8} are equivalent in the following sense:
\smallskip

Given $\{u_0,f,h_1,h_2,h_3\}$ one can find $\{y_0,f,g_1,g_2,g_3\}$ such that the corresponding solution $u$ of \eqref{w6} is exactly the same as the corresponding $y$ for the system \eqref{w8} and vice versa.

\medskip
In fact, for given $u_0\in L^2(0,L)$, $f\in L^1(0,T;L^2(0,L))$ and $\vec{h}\in\mathcal{H}_T$, system \eqref{w6} admits a unique solution $u\in\mathcal{Z}_T$. Let $y_0=u_0$ and set
$$g_1(t)=h_1(t), \text{ \ \ }g_3(t)=h_2(t), \text{ \ \ }g_2(t)=h_3(t).$$
Then, according to Proposition \ref{l3}, we have $\vec{g}\in\mathcal{G}_t$. Due to the uniqueness of IBVP \eqref{w8}, with the selection $\{y_0,f,g_1,g_2,g_3\}$, the corresponding solution $y\in\mathcal{Z}_T$ of \eqref{w8} must be equal to $u$ since $u$ also solves \eqref{w8} with the given auxiliary data $\{y_0,f,g_1,g_2,g_3\}$. On the other hand, for any given $y_0\in L^2(0,L)$, $f\in L^1(0,T;L^2(0,L))$ and $\vec{g}\in\mathcal{G}_T$, let $y\in\mathcal{Z}_T$ be the corresponding solution of the system \eqref{w8}. From \eqref{w9}, we have $y_{xx}(0,\cdot)$ and $y_{xx}(L,\cdot)\in H^{-\frac{1}{3}}(0,T)$. Thus, if set $u_0=y_0$ and
$$h_1(t)=u_{xx}(0,t), \text{ \ \ }h_2(t)=g_3(t), \text{ \ \ }h_3(t)=u_{xx}(L,T),$$
then $\vec{h}\in\mathcal{H}_T$ and the corresponding solution $u\in\mathcal{Z}_T$ of \eqref{w6} must be equal to $y$ which also solves \eqref{w6} with the auxiliary data $(u_0,f,\vec{h})$.
\end{remark}

Finally, we are at the stage to prove the well-posedness of the  of the following nonlinear system,
\begin{equation}
\left\{
\begin{array}
[c]{lll}%
u_t+u_x+uu_x+u_{xxx}=0 &  & \text{ in } (0,L)\times(0,T),\\
u_{xx}(0,t)=h_1(t),\text{ }u_x(L,t)=h_2(t),\text{ }u_{xx}(L,t)=h_3(t) &  & \text{ in }(0,T),\\
u(x,0)=u_0(x):=\phi(x)  & & \text{ in }(0,L).
\end{array}
\right.  \label{3.1}
\end{equation}
For given $T>0$, define
\[ X_{T}:= L^2  (0,L)\times H^{-\frac{1}{3}}(0,T)\times
L^2  (0,T)\times H^{-\frac{1}{3}}(0,T)\]
and
\[
\mathcal{Z}_T:=C([0,T];L^2 (0,L))\cap L^2 (0,T; H^{1}(0,L). \]



\begin{theorem} Let $T>0$ and $r>0$ be given. There exists a $T^*\in
(0,T]$ such that for any $(\phi , \vec{h}) \in X_{
T}$ with
\[ \| (\phi, \vec{h})\|_{X_T} \leq r, \] the IBVP (\ref{3.1}) admits a unique solution
\[ u\in \mathcal{Z}_{T^*}.\]
In addition, the solution $u$  possesses the hidden regularities
\[ \partial ^k_x u\in L^ {\infty}_x (0,L; H^{\frac{1-k}{3}} (0,T^*)), \quad k=0,1,2\]
and, moreover, the corresponding solution map is Lipschitz continuous.
\end{theorem}
\begin{proof} 
Since the proof is similar to that  presented in \cite{BonaSunZhang,KrZh}, we will omit it.
\end{proof}

\section{Boundary controllability\label{Sec3}}
In this section, we   study exact  boundary controllability  of the system
\begin{equation}
\left\{\begin{array}[c]{lll}%
u_t+u_x+uu_x+u_{xxx}=0 &  & \text{ in } (0,L)\times(0,T),\\
u_{xx}(0,t)=0,\text{ }u_x(L,t)=h(t),\text{ }u_{xx}(L,t)=0 &  & \text{ in }(0,T),\\
u(x,0)=u_0(x) & & \text{ in }(0,L),
\end{array}
\right.  \label{a1.1-1}
\end{equation}
around a constant steady state $u\equiv c$.   As it is easy to see by letting  $u=v+c$, it is equivalent to study the exact boundary controllability of the following system
\begin{equation}
\left\{\begin{array}[c]{lll}%
v_t+(c+1)v_x+vv_x+v_{xxx}=0 &  & \text{ in } (0,L)\times(0,T),\\
v_{xx}(0,t)=0,\text{ }v_x(L,t)=h(t),\text{ }v_{xx}(L,t)=0 &  & \text{ in }(0,T)\\
v(x,0)=v_0(x) & & \text{ in }(0,L),
\end{array}
\right.  \label{a1.1-2}
\end{equation}
around the origin $0$  instead.
We have the following exact controllability results for the system (\ref{a1.1-2}).
\begin{theorem}
Let $T>0$, $c+1\neq0$ and
\begin{equation}
L\notin\mathcal{R}_c:=\left\{  \frac{2\pi}{\sqrt{3(c+1)}}\sqrt{k^{2}+kl+l^{2}}
\,:k,\,l\,\in\mathbb{N}^{\ast}\right\}\cup\left\{\frac{k\pi}{\sqrt{c+1}}:k\in\mathbb{N}^{\ast}\right\}
\label{critical_c}
\end{equation}
be given. Then there exists a $\delta>0$ such that for any
$v_0,v_T\in L^2(0,L)$ with
$$||v_0||_{L^2(0,L)}+||v_T||_{L^2(0,L)}\leq\delta,$$
one can find $h\in L^2(0,T)$ such that the system  (\ref{a1.1-2})
admits unique solution
$$v\in C([0,T];L^2(0,L))\cap L^2(0,T;H^1(0,L))$$
satisfying
$$v(x,0)=v_0(x), \text{  } \text{  } \text{  }v(x,T)=v_T(x)  \text{  } \text{ in } (0,L).$$
\label{th7_int}
\end{theorem}

 To prove the Theorem \ref{th7_int},  we first consider  the linear system associated to (\ref{a1.1-2})
 \begin{equation}
\left\{\begin{array}[c]{lll}%
v_t+(c+1)v_x +v_{xxx}=0 &  & \text{ in } (0,L)\times(0,T),\\
v_{xx}(0,t)=0,\text{ }v_x(L,t)=h(t),\text{ }v_{xx}(L,t)=0 &  & \text{ in }(0,T),\\
v(x,0)=v_0(x) & & \text{ in }(0,L)
\end{array}
\right.  \label{EBC_1}
\end{equation}
 and its adjoint system
 \begin{equation}
\left\{
\begin{array}
[c]{lll}%
\psi_t+(c+1)\psi_x+\psi_{xxx}=0&  &\text{ in }(0,L)\times(0,T),\\
(c+1)\psi(0,t)+\psi_{xx}(0,t)=0&  &\text{ in }(0,T),\\
(c+1)\psi(L,t)+\psi_{xx}(L,t)=0&  &\text{ in }(0,T),\\
\psi_x(0,t)=0&  &\text{ in }(0,T),\\
\psi(x,T)=\psi_T(x) & &\text{ in }(0,L).
\end{array}
\right.  \label{a1}
\end{equation}
Note that by transformation $x'=L-x$ and $t'=T-t$,  the system (\ref{a1}) is equivalent of the following forward system
\begin{equation}
\left\{
\begin{array}
[c]{lll}%
\varphi_t+(c+1)\varphi_x+\varphi_{xxx}=0 &  &\text{ in }(0,L)\times(0,T),\\
(c+1)\varphi(0,t)+\varphi_{xx}(0,t)=0&  &\text{ in }(0,T),\\
(c+1)\varphi(L,t)+\varphi_{xx}(L,t)=0&  &\text{ in }(0,T),\\
\varphi_x(L,t)=0&  &\text{ in }(0,T),\\
\varphi(x,0)=\varphi_0(x) & &\text{ in }(0,L).
\end{array}
\right.  \label{a2}
\end{equation}

\begin{proposition}\label{sharp}
For any $\varphi_0\in L^2(0,L)$,  the system \eqref{a2} admits a unique solution $\varphi\in\mathcal{Z}_T$ which, moreover,   possesses the following hidden  regularities
\begin{equation}
\sup_{x\in(0,L)}||\partial^r_x\varphi(x,\cdot)||_{H^{\frac{1-r}{3}}(0,T)}\leq C_r||\varphi_0||_{L^2(0,L)},
\label{trace}
\end{equation}
for $r=0,1,2$.
\end{proposition}
\begin{remark}
Equivalently, for any $\psi _T\in L^2 (0,L)$,   the system \eqref{a1} admits a unique solution $\psi \in Z_T$ which, moreover, possesses  the hidden regularities
\begin{equation}
\sup_{x\in(0,L)}||\partial^r_x\psi(x,\cdot)||_{H^{\frac{1-r}{3}}(0,T)}\leq C_r||\psi_T||_{L^2(0,L)},
\label{trace_a}
\end{equation}
for $r=0,1,2$.
\end{remark}
\begin{proof}[\bf{Proof of Proposition \ref{sharp}}]
Let us consider the set $$\mathcal{X}_T:=\{u\in\mathcal{Z}_T:\partial^r_xu\in L_x^{\infty}(0,L;H^{\frac{1-r}{3}}(0,T)), r=0,1,2\}$$ which is a Banach space equipped with the norm
$$||u||_{\mathcal{X}_T}:=||u||_{\mathcal{Z}_T}+\sum_{r=0}^2||\partial^r_xu||_{L_x^{\infty}(0,L;H^{\frac{1-r}{3}}(0,T))}.$$
According to Proposition \ref{p3}, for any $v\in\mathcal{X}_{\beta}$ where $0<\beta\leq T$ and any $\varphi_0\in L^2 (0,L)$, the system
\begin{equation}
\left\{
\begin{array}
[c]{lll}%
w_t+w_{xxx}=-(c+1)v_x  &  &\text{ in }(0,L)\times(0,T),\\
w_{xx}(0,t)=-(c+1)v(0,t)&  &\text{ in }(0,T),\\
w_{xx}(L,t)=-(c+1)v(L,t)&  &\text{ in }(0,T),\\
w_x(L,t)=0 &  &\text{ in }(0,T),\\
w(x,0)=\varphi_0(x) &  &\text{ in }(0,L),
\end{array}
\right.  \label{a4}
\end{equation}
admits a unique solution $w\in\mathcal{X}_{\beta}$ and, moreover,
$$||w||_{\mathcal{X}_{\beta}}\leq C\left(||\psi_0||_{L^2(0,L)}+||v(0,\cdot)||_{H^{-\frac{1}{3}}(0,T)}+||v(L,\cdot)||_{H^{-\frac{1}{3}}(0,T)}+||v_x||_{L^1(0,\beta;L^2(0,L))}\right),$$
where the constant $C>0$ depends only on $T$. As we have,
$$||v_x||_{L^1(0,\beta;L^2(0,L))}\leq C\beta^{1/2}||v||_{\mathcal{X}_{\beta}},$$
$$||v(0,\cdot)||_{H^{-\frac{1}{3}}(0,\beta)}\leq ||v(0,\cdot)||_{L^2(0,\beta)}\leq \beta^{2/3}||v(0,\cdot)||_{L^6(0,\beta)}\leq C\beta^{2/3}||v(0,\cdot)||_{H^{\frac{1}{3}}(0,\beta)}\leq C\beta^{2/3}||v||_{\mathcal{X}_{\beta}}$$
and
$$||v(L,\cdot)||_{H^{-\frac{1}{3}}(0,\beta)}\leq ||v(L,\cdot)||_{L^2(0,\beta)}\leq \beta^{2/3}||v(L,\cdot)||_{L^6(0,\beta)}\leq C\beta^{2/3}||v(L,\cdot)||_{H^{\frac{1}{3}}(0,\beta)}\leq C\beta^{2/3}||v||_{\mathcal{X}_{\beta}},$$
the system \eqref{a4} defines a map as follows
$$
\begin{array}{lcl} \Gamma: &\mathcal{X}_{\beta} \longrightarrow & \mathcal{X}_{\beta}\\
&v\mapsto&\Gamma(v)=w,\\
\end{array}
$$
for any $v\in\mathcal{X}_T$ and $\beta\in(0,\max\{1,T\}]$. Here $w\in\mathcal{X}_{\beta}$ is the corresponding solution of \eqref{a4} and
$$||\Gamma(v)||_{\mathcal{X}_{\beta}}\leq C_1||\psi_0||_{L^2(0,L)}+C_2\beta^{1/2}||v||_{\mathcal{X}_{\beta}},$$
where $C_1$ and $C_2$ are two positive constants depending only on $T$. Choosing $r>0$ and $\beta\in\left(0,\max\{1,T\}\right]$ such that
$$r=2C_1||\psi_0||_{L^2(0,L)}$$
and
$$2C_2\beta^{1/2}\leq\frac{1}{2},$$
then, for any
$$v\in \mathcal{B}_{\beta,r}=\{v\in\mathcal{X}_{\beta}: ||v||_{\mathcal{X}_{\beta}}\leq r\},$$
we have
$$||\Gamma(v)||_{\mathcal{X}_{\beta}}\leq r.$$
Moreover, for any $v_1,v_2\in\mathcal{B}_{\beta,r}$, we get
$$||\Gamma(v_1)-\Gamma(v_2)||_{\mathcal{X}_{\beta}}\leq 2C_2\beta^{1/2}||v_1-v_2||_{\mathcal{X}_{\beta}}\leq\frac{1}{2}||v_1-v_2||_{\mathcal{X}_{\beta}}.$$
Therefore, the map $\Gamma$ is a contraction mapping on $\mathcal{B}_{\beta,r}$. Its fixed point $w=\Gamma(v)\in\mathcal{X}_{\beta}$ is the desired solution for $t\in(0,\beta)$. As the chosen $\beta$ is independent of $\varphi_0$, the standard continuation extension argument yields that the solution $w$ belongs to $\mathcal{X}_{\beta}$. The proof is complete.
\end{proof}
The system  \eqref{a2} possesses  an elementary estimate as described below.
\begin{proposition}\label{p_adj_1}
Any solution $\varphi$ of the adjoint system \eqref{a2} with initial data $\varphi_0\in L^2(0,L)$ satisfies
\begin{equation}
||\varphi_0||^2_{L^2(0,L)}\leq \frac{1}{T}||\varphi||^2_{L^2((0,L)\times(0,T))}+||\varphi_x(0,\cdot)||^2_{L^2(0,T)}+(c+1)||\varphi(0,\cdot)||^2_{L^2(0,T)}.
\label{a5}
\end{equation}
\end{proposition}
\begin{proof}
Multiplying the equation \eqref{a2} by $(T-t)\varphi$ and integrating by parts over $(0,L)\times(0,T)$, we get
$$\frac{T}{2}\int_0^L\varphi^2_0dx=\frac{1}{2}\int_0^L\int_0^T\varphi^2dxdt+\int_0^T\left(\frac{T-t}{2}\right)\left(-(c+1)\varphi^2(L)+(c+1)\varphi^2(0)+\varphi_x^2(0)\right)dt,$$
which yields \eqref{a5} since $c+1>0$.
\end{proof}
Equivalently, the following estimate holds for solutions $\psi$ of the system \eqref{a1}
\begin{equation}
||\psi_T||^2_{L^2(0,L)}\leq \frac{1}{T}||\psi||^2_{L^2((0,L)\times(0,T))} +(c+1)||\psi(L,\cdot)||^2_{L^2(0,T)} +||\psi _x(L,\cdot)||^2_{L^2(0,T)}.
\label{a6}
\end{equation}
\begin{remark}\label{rosier}
As a comparison, it is worth pointing out that for the adjoint system of \eqref{w8}, which is given by
\begin{equation}
\left\{
\begin{array}
[c]{lll}%
\xi_t+\xi_x+\xi_{xxx}=0  &  &\text{ in }(0,L)\times(0,T),\\
\xi(0,t)=0,\text{ }\xi(L,t)=0,\text{ }\xi_x(0,t)=0  &  &\text{ in }(0,T),\\
\xi(x,T)=\xi_T(x) & &\text{ in }(0,L),
\end{array}
\right.  \label{a7}
\end{equation}
the following inequality holds
\begin{equation}
||\xi_T||_{L^2(0,L)}\leq \frac{1}{T}||\xi||_{L^2((0,L)\times(0,T))}+||\xi_x(L,\cdot)||^2_{L^2(0,T)}.
\label{a8}
\end{equation}
The extra term $||\psi(L,\cdot)||^2_{L^2(0,T)}$ in \eqref{a6} brings a  technique difficulty in establishing the observability inequality of the adjoint system \eqref{a1}, which calls the use of  the hidden regularities  established in Proposition \ref{sharp}.
\end{remark}

Now we turn to analyze the exact controllability of the linear system \eqref{EBC_1}.




\begin{proposition}\label{control}
Assume $c+1\neq0$. 
Let $T>0$ and $L\notin\mathcal{R}_c$ be given. There exists a bounded linear operator
 $$
\begin{array}{lcl} \Psi: &L^2(0,L)\times L^2(0,L) \longrightarrow & L^2(0,T)
\end{array}
$$
such that for any $v_0,v_T\in L^2(0,L)$, if one chooses $h_2=\Psi(v_0,v_T)$, then system \eqref{EBC_1} admits a solution $v\in\mathcal{Z}_T$ satisfying
\[ v|_{t=0} =v_0, \qquad v|_{t=T} = v_T.\]
\end{proposition}

\begin{proof} It suffices to prove that:

\smallskip
{\em For any given  $L\in(0,+\infty)\backslash\mathcal{R}_c$ and $T>0$ , there exists  a positive constant $C$ depending only on $T$ and $L$ such  that
\begin{equation}
||\psi_T||_{L^2(0,L)}\leq C||\psi_x(L,t)||_{L^2(0,T)}
\label{obs1}
\end{equation}
holds for any $\psi_T\in L^2(0,L)$, where $\psi$ is the solution of \eqref{a1} with the terminal  data $\psi_T$}.

\medskip
We proceed by contradiction as in \cite[Proposition 3.3]{Rosier}. If \eqref{obs1} does not holds, then there exists a sequence $\{\psi^n_T\}_{n\in\mathbb{N}}\in L^2(0,L)$ with
\begin{equation}
||\psi^n_T||_{L^2(0,L)}=1, \forall n\in\mathbb{N}
\label{c2}
\end{equation}
such that the corresponding solutions of \eqref{a1} satisfy
\begin{equation}
1=||\psi^n_T||_{L^2(0,L)}>n||\psi^n_x(L,t)||_{L^2(0,T)},
\label{c3}
\end{equation}
thus $||\psi^n_x(L,t)||_{L^2(0,T)}\to0$, as $n\to\infty$. Thanks to Proposition \ref{sharp} we have $\{\psi^n\}_{n\in\mathbb{N}}$ is bounded in $L^2(0,T;H^1(0,L))$ and $\{\psi^n(0,t)\}_{n\in\mathbb{N}}$
is bounded in $H^{\frac{1}{3}}(0,T)$. In addition, according to Proposition \ref{p_adj_1}, we have
\begin{equation}
||\psi^n_T||_{L^2(0,L)}\leq \frac{1}{T}||\psi^n||^2_{L^2((0,L)\times(0,T))}+||\psi^n_x(L,\cdot)||^2_{L^2(0,T)}+(c+1)||\psi^n(0,\cdot)||^2_{L^2(0,T)}.
\label{c4}
\end{equation}
Since $\psi^n_t=-(c+1)\psi^n_x-\psi^n_{xxx}$ is bounded in $L^2(0,T;H^{-2}(0,L))$ and the embedding
$$H^1(0,L)\hookrightarrow L^2(0,L)\hookrightarrow H^{-2}(0,L),$$
the sequence $\{\psi^n\}_{n\in\mathbb{N}}$ is relatively compact in $L^2(0,T;L^2(0,L))$ (see \cite{Simon}). Furthermore,
the second term on the right in \eqref{c4} converges to zero in $L^2(0,T)$, and by the compact embedding
$$H^{\frac{1}{3}}(0,T)\hookrightarrow L^2(0,T),$$
the sequence $\{\psi^n(0,t)\}_{n\in\mathbb{N}}$ has a convergent subsequence in $L^2(0,T)$. Therefore by \eqref{c4}, $\{\psi^n_T\}_{n\in\mathbb{N}}$ is an $L^2(0,L)$--Cauchy sequence, thus, at least for a subsequence, we have
\begin{equation}
\psi^n_T\longrightarrow\psi_T  \text{ in } L^2(0,L),
\label{c5}
\end{equation}
by Theorem \ref{sharp} holds that,
\begin{equation}
\psi^n_x(L,t)\longrightarrow\psi_x(L,t)  \text{ in } L^2(0,T).
\label{c6}
\end{equation}
From \eqref{c2}, \eqref{c5} and \eqref{c6}, we have $\psi$ is a solution of
\begin{equation}
\left\{
\begin{array}
[c]{lll}
\psi_t+(c+1)\psi_x+\psi_{xxx}=0 &  & \text{in }(0,L)\times(0,T),\\
(c+1)\psi(0,t)+\psi_{xx}(0,t)=0&  & \text{in }(0,T),\\
(c+1) \psi(L,t)+\psi_{xx}(L,t)=0&  & \text{in }(0,T),\\
\psi_x(0,t)=0&  & \text{in }(0,T),
\end{array}
\right.  \label{c7}
\end{equation}
satisfying the additional boundary condition
\begin{equation}
\psi_x(L,t)=0
\label{c8}
\end{equation}
and
\begin{equation}
||\psi_T||_{L^2(0,L)}=1.
\label{c9}
\end{equation}
Notice that \eqref{c9} implies that the solutions of \eqref{c7}-\eqref{c8} can not be identically zero. Therefore, by the following Lemma \ref{ucp1}, one can conclude that $\psi\equiv0$, therefore, $\psi_T(x)\equiv0$, which contradicts \eqref{c9}.
\end{proof}

\begin{lemma}
\label{ucp1}For any $T>0$, let $N_T$ denote the space of the initial states $\psi_T\in L^2(0,L)$ such that the mild solution $\psi$ of \eqref{c7} satisfies \eqref{c8}. Then, for $L\in(0,+\infty)\backslash\mathcal{R}_c$, $N_T=\{0\}, \forall
T>0$.
\end{lemma}

\begin{proof}
The proof uses the same arguments as those given in \cite{Rosier}. Therefore, If $N_T\neq\left\{  0\right\}  $, the map $\psi_T\in\mathbb{C}N_T\longrightarrow A(\psi_T)\in\mathbb{C}N_T$ (where $\mathbb{C}N_T$ denote the complexification of $N_T$) has (at least) one eigenvalue, hence, there exists $\lambda\in\mathbb{C}$ and $\psi_0\in H^3(0,L)\backslash\{0\} $ such that
\begin{equation}
\left\{
\begin{array}
[c]{l}
\lambda\psi_0= -(c+1)\psi_0^{\prime}-\psi_0^{\prime\prime\prime},\\
(c+1)\psi_0(0)+\psi_0^{\prime\prime}(0)=0, \text{ \ }(c+1)\psi_0(L)+\psi_0^{\prime\prime}(L)= 0, \text{ \ }\psi_0^{\prime}(0)=0, \text{ \ }\psi_0^{\prime}(L)=0.
\end{array}
\right.  \label{c10}
\end{equation}
To conclude the proof of the Lemma \ref{ucp1}, we prove that this does not hold if $L\notin\mathcal{R}_c$. To simplify the notation, henceforth we denote $\psi_0:=\psi$.

\begin{lemma}
\label{cont1}Let $L>0$. Consider the assertion
$$(\mathcal{F})\text{ \ \  \ \  \ \  \ \  \ \  \ \ }\exists\lambda\in\mathbb{C}\text{, }\exists\psi\in H^3(0,L)\backslash\{0\}\text{ such that }
\left\{
\begin{array}
[c]{l}
\lambda\psi=-(c+1)\psi^{\prime}-\psi^{\prime\prime\prime},\\
(c+1)\psi(0)+\psi^{\prime\prime}(0)=0,\\
(c+1)\psi(L)+\psi^{\prime\prime}(L)= 0,\\
\psi^{\prime}(0)=0, \text{ \ }\psi^{\prime}(L)=0.
\end{array}
\right.
$$
Then, $(   \mathcal{F})  $ holds if and only if $L\in\mathcal{R}_c$.
\end{lemma}
\noindent\textit{Proof. }
We will use the argument developed in \cite[Lemma 3.5]{Rosier}. Assume that $\psi$ satisfies $\mathcal{F}$. Let us introduce the notation $\hat{\psi}(\xi)=\int_0^L\psi(\xi)e^{-ix\xi}dx$. Then, multiplying the equation \eqref{c10} by $e^{-ix\xi}$, integrating by parts in $(0,L)$ and using the boundary condition we obtain
\begin{equation}
(\lambda+(c+1)(i\xi)+(i\xi)^{3})\hat{\psi}(\xi)= (i\xi)^2\psi(0)-(i\xi)^2\psi(L)e^{-iL\xi} .
\label{c11}
\end{equation}
Setting $\lambda=-ip$, we have
\begin{equation}
\hat{\psi}(\xi)=-i\xi^2\frac{\alpha-\beta e^{-iL\xi}}{\xi^{3}-(c+1)\xi+p}
\label{c12}
\end{equation}
with
\[ \alpha=\psi (0), \qquad \beta =\psi (L).\]

Using Paley-Wiener theorem (see \cite[Section 4, page 161]{yosida}) and the usual characterization of $H^2(\mathbb{R})$ by means of the Fourier transform we see that $\mathcal{F}$ is equivalent to the existence of $p\in\mathbb{C}$ and
$$(\alpha,\beta)\in\mathbb{C}^2\backslash\{(0,0)\},$$
such that
$$f(\xi):=\xi^2\frac{\alpha-\beta e^{-iL\xi}}{\xi^{3}-(c+1)\xi+p}$$
satisfies

a) $f$ is entire function in $\mathbb{C}$;

b) ${\displaystyle\int_{\mathbb{R}}}|f(\xi)|^2(1+|\xi|^2)^2d\xi<\infty$ ;

c) $\forall\xi\in\mathbb{C}$, we have that $|f(\xi)|\leq c_1(1+|\xi|)^ke^{L|\operatorname{Im}\xi|}$ for some positive constants $c_1$ and $k$.

Recall that $f$ is a entire function if only if, the roots $\xi_0,\xi_1,\xi_2$ of $Q(\xi):=\xi^3-(c+1)\xi+p$ are roots of
\begin{equation}
s(\xi):=\xi^2(\alpha-\beta e^{-iL\xi}).
\label{c13}
\end{equation}
In addition, all the roots of $\alpha-\beta e^{-iL\xi}$ are simple, otherwise $\alpha =\beta =0$ which implies that $\psi(0)=\psi(L)=0$. Using the system \eqref{c10} we conclude that $\psi\equiv0$. Besides, as $c+1\neq0$, the three roots of $Q(\xi)$ must be simple too.
\vglue 0.2 cm
Let us first assume that $Q(\xi)$ and $\alpha-\beta e^{-iL\xi}$ share the same roots, we can write the three roots of $Q(\xi)$ as
\begin{equation}
\displaystyle\xi_1:=\xi_0+k\frac{2\pi}{L} \text{ \ and \ } \displaystyle\xi_2:=\xi_1+l\frac{2\pi}{L}
\label{c14}
\end{equation}
with $k$ and $l$ are some positive integers, we have
\begin{equation}
Q(\xi)=(\xi-\xi_0)(\xi-\xi_1)(\xi-\xi_2),
\label{c15}
\end{equation}
that is
\begin{equation}
\left\{
\begin{array}{rrl}
\xi_0+\xi_1+\xi_2 & = & 0 \\
\\
\xi_0\xi_1+\xi_0\xi_2+\xi_1\xi_2 & = & -(c+1) \\
\\
\xi_0\xi_1\xi_2 & = & -p.
\end{array}
\right.
\label{c16}
\end{equation}
Thus we have
\begin{equation}
\left\{
\begin{array}{rrl}
L & = & 2\displaystyle\pi\sqrt{\frac{k^2+kl+l^2}{3(1+c)}} \\
\\
\xi_0 & = & -\displaystyle\frac{1}{3}(2k+l)\frac{2\pi}{L}\\
\\
p & = & -\xi_0\displaystyle\left(\xi_0+k\frac{2\pi}{L}\right)\left(\xi_0+(k+l)\frac{2\pi}{L}\right).
\end{array}
\right.
\label{c17}
\end{equation}
Next we assume that $\xi=0$ is a root of $Q(\xi)$, but not of $\alpha-\beta e^{-iL\xi}$ . Then three roots of $Q(\xi)$ can be written as $0$, $\xi_1$, $\xi_1+k\frac{2\pi}{L}$ with $k$ being a positive integer. We have
\begin{equation}
\left\{
\begin{array}{rrl}
\xi_1+\xi_2 & = & 0 \\
\\
\xi_1\xi_2 & = & -(c+1) \\
\\
0 & = & -p,
\end{array}
\right.
\label{c16a}
\end{equation}
and, consequently, follows that
\begin{equation}
\left\{
\begin{array}{rrl}
L & = &\displaystyle\frac{k\pi}{\sqrt{(1+c)}} \\
\\
\xi_1 & = & -k\displaystyle\frac{\pi}{L}\\
\\
p & = & 0.
\end{array}
\right.
\label{c17a}
\end{equation}
Hence, $\mathcal{F}$ holds if and only if $L\in\mathcal{R}_c$. This complete the proof of Lemma \ref{cont1} and, consequently, the proof of Lemma \ref{ucp1}.
\end{proof}

Finally we consider the case of $c+1=0$. Then it is easy to see that $\xi=0$ must be a root of $Q(\xi)$; otherwise $L=\infty$. Hence
$$f(\xi):=\frac{\alpha-\beta e^{-iL\xi}}{\xi}.$$
We must have
$$\alpha=\beta \text{  }\text{ or }\text{ }\psi(0)=\psi(L).$$
\eqref{c10} becomes
\begin{equation}
\left\{
\begin{array}
[c]{l}
\psi_0^{\prime\prime\prime}=0,\\
\psi_0^{\prime\prime}(0)=0, \text{ \ }\psi_0^{\prime\prime}(L)= 0, \text{ \ }\psi_0^{\prime}(0)=0, \text{ \ }\psi_0^{\prime}(L)=0, \text{ \ } \psi_0(0)=\psi_0(L)
\end{array}
\right.  \label{c10a}
\end{equation}
which implies that $\psi_0(x)\equiv C$.

\medskip

We are now ready to present the proof of  Theorem \ref{th7_int}.

\begin{proof}[\bf{Proof of Theorem \ref{th7_int}}]
Rewrite the system \eqref{a1.1-2} in its integral form
\begin{equation}
v(t)=W_0(t)v_0+W_{bdr}(t)h-\int_0^tW_0(t-\tau)(vv_x)(\tau,x)d\tau.
\label{n_lin1}
\end{equation}
For any $u\in\mathcal{Z}_T$, let us define
$$\nu(T,u):=\int_0^TW_0(T-\tau)(uu_x)d\tau.$$
By using Proposition \ref{control},  for any $v_0,v_T\in L^2(0,L)$,  if we choose
$$h=\Psi(v_0,v_T+\nu(T,u)),$$
then
$$u(t)=W_0(t)v_0+W_{bdr}\Psi(v_0,v_t+\nu(T,u))-\int_0^tW_0(t-\tau)(uu_x)(\tau,x)d\tau$$
satisfies
$$u(x,0)=v_0(x), \text{ } \text{ } \text{ } \text{ } \text{ } u(x,T)=v_T(x)+\nu(T,v)-\nu(T,v)=v_T(x).$$
This leads us to consider the map
$$\Gamma(u)  =W_0(t)v_0+W_{bdr}\Psi(v_0,v_t+\nu(T,u))-\int_0^tW_0(t-\tau)(uu_x)(\tau,x)d\tau.$$
If we can show that the map $\Gamma$ is a contraction in an appropriate metric space, then its fixed point $u$ is a solution of \eqref{a1.1-2} with $h=\Psi(v_0,v_T+\nu(T,u))$ that satisfies
$$u(x,0)=v_0(x), \text{ } \text{ } \text{ } \text{ } \text{ } u(x,T)=v_T(x).$$

 Let $$B_r=\{z\in\mathcal{Z}_T:||z||_{\mathcal{Z}_T}\leq r\}.$$
By  Proposition \ref{p3}, there exists a constant $C_1>0$ such that  for any $u\in\mathcal{Z}_T$,
$$||\Gamma(u)||_{\mathcal{Z}_T}\leq C_1\left(||v_0||_{L^2(0,L)}+||\Psi(v_0,v_t+\nu(T,u))||_{L^2(0,L)}-\int_0^T||uu_x||_{L^2(0,L)}(t)dt\right).$$
Furthermore,  as
$$||\Psi(v_0,v_t+\nu(T,u))||_{L^2(0,L)}\leq C_2\left( ||v_0||_{L^2(0,L)}+||v_T||_{L^2(0,L)}+||\nu(T,u)||_{L^2(0,L)} \right)$$
and
$$||\nu(T,u)||_{L^2(0,L)}\leq\int_0^T||uu_x||_{L^2(0,L)}(t)dt\leq C_3||u||^2_{\mathcal{Z}_T},$$
we infer that
$$||\Gamma(u)||_{\mathcal{Z}_T}\leq C_3\left(||v_0||_{L^2(0,L)}+||v_T||_{L^2(0,L)}\right)+C_4||u||^2_{\mathcal{Z}_T},$$
for any $u\in\mathcal{Z}_T$ where $C_3$ and $C_4$ are constants depending only $T$.  Thus, if we select  $r$ and $\delta$ satisfying
$$r=2C_3\delta$$
and
$$4C_3C_4\delta<\frac{1}{2},$$
then  the operator $\Gamma$ maps $B_r$ into itself for any $v\in B_r$.
In addition, for  any $\tilde{u}, u\in B_r$, the similar arguments yield that
$$||\Gamma(u)-\Gamma(\tilde{u})||_{\mathcal{Z}_T}\leq\gamma||u-\tilde{u}||_{\mathcal{Z}_T}$$
with $\gamma=8C_3C_4\delta<1$. Therefore the map $\Gamma$ is a contraction. Its fixed point is a desired solution. The proof  of Theorem \ref{th7_int} is completed and, consequently, Theorem \ref{th1.2} follows.
\end{proof}

\section{Multi controls and null controllability\label{Sec4}}

In this section we will first consider the following linear systems associated to \eqref{m1}:
\begin{equation}
\left\{\begin{array}[c]{lll}%
u_t+u_x+u_{xxx}=0 &  & \text{ in } (0,L)\times(0,T),\\
u_{xx}(0,t)=h_1(t),\text{ }u_x(L,t)=h_2(t),\text{ }u_{xx}(L,t)=0 &  & \text{ in }(0,T),\\
u(x,0)=u_0(x) & & \text{ in }(0,L),
\end{array}
\right.  \label{EBC_1c}
\end{equation}
\begin{equation}
\left\{\begin{array}[c]{lll}%
u_t+u_x+u_{xxx}=0 &  & \text{ in } (0,L)\times(0,T),\\
u_{xx}(0,t)=0,\text{ }u_x(L,t)=h_2(t),\text{ }u_{xx}(L,t)=h_3(t) &  & \text{ in }(0,T),\\
u(x,0)=u_0(x) & & \text{ in }(0,L)
\end{array}
\right.  \label{EBC_1d}
\end{equation}
and
\begin{equation}
\left\{\begin{array}[c]{lll}%
u_t+u_x+u_{xxx}=0 &  & \text{ in } (0,L)\times(0,T),\\
u_{xx}(0,t)=h_1(t),\text{ }u_x(L,t)=0,\text{ }u_{xx}(L,t)=h_3(t) &  & \text{ in }(0,T),\\
u(x,0)=u_0(x) & & \text{ in }(0,L).
\end{array}
\right.  \label{EBC_1e}
\end{equation}

\begin{proposition}\label{control_b}
Let $T>0$ and $L>0$ be given. There exists a bounded linear operator
$$
\begin{array}{lcl}
 \Theta: &L^2(0,L)\times L^2(0,L) \longrightarrow & H^{-\frac{1}{3}}(0,T)\times L^2(0,T)
\end{array}
$$
such that for any $u_0,u_T\in L^2(0,L)$, if one chooses $$(h_1,h_2)=\Psi(u_0,u_T),$$ then system \eqref{EBC_1c}
admits a solution $u\in\mathcal{Z}_T$ satisfying \[ u(x,0)=u_0 (x), \qquad u(x,T)= u_T (x).\]
\end{proposition}
\begin{proposition}\label{control_c}
Let $T>0$ and $L>0$ be given. There exists a bounded linear operator
$$
\begin{array}{lcl}
 \Pi: &L^2(0,L)\times L^2(0,L) \longrightarrow & L^2(0,T)\times H^{-\frac{1}{3}}(0,T)
\end{array}
$$
such that for any $u_0,u_T\in L^2(0,L)$, if one chooses $$(h_2,h_3)=\Psi(u_0,u_T),$$ then system \eqref{EBC_1d}
admits a solution $u\in\mathcal{Z}_T$ satisfying\[ u(x,0)=u_0 (x), \qquad u(x,T)= u_T (x).\]
\end{proposition}
\begin{proposition}\label{control_d}
Let $T>0$ and $L>0$ be given. There exists a bounded linear operator
$$
\begin{array}{lcl}
 \Lambda: &L^2(0,L)\times L^2(0,L) \longrightarrow & H^{-\frac{1}{3}}(0,T)\times H^{-\frac{1}{3}}(0,T)
\end{array}
$$
such that for any $u_0,u_T\in L^2(0,L)$, if one chooses $$(h_1,h_3)=\Psi(u_0,u_T),$$ then system \eqref{EBC_1e} admits a solution $u\in\mathcal{Z}_T$ satisfying \[ u(x,0)=u_0 (x), \qquad u(x,T)= u_T (x).\]
\end{proposition}
Propositions \ref{control_b}-\ref{control_d} follow as a consequence of the following observability
inequalities for the solution of the backward system \eqref{5}
\begin{equation}
||\psi_T||_{L^2(0,L)}\leq C\left(||\Delta^{\frac{1}{3}}_t\psi(0,t)||_{L^2(0,T)}+||\psi_x(L,t)||_{L^2(0,T)}\right),
\label{obs3}
\end{equation}
\begin{equation}
||\psi_T||_{L^2(0,L)}\leq C\left(||\psi_x(L,t)||_{L^2(0,T)}+||\Delta^{\frac{1}{3}}_t\psi(L,t)||_{L^2(0,T)}\right)
\label{obs4}
\end{equation}
and
\begin{equation}
||\psi_T||_{L^2(0,L)}\leq C\left(||\Delta^{\frac{1}{3}}_t\psi(0,t)||_{L^2(0,T)}+||\Delta^{\frac{1}{3}}_t\psi(L,t)||_{L^2(0,T)}\right),
\label{obs5}
\end{equation}
where $\Delta_t:=(I-\partial_t^2)^{\frac{1}{2}}$. The proofs of \eqref{obs3}-\eqref{obs5} are similar to  that  of \eqref{obs1}. Furthermore, Theorem \ref{th3_int} can be proved using the same arguments as  that in the proof of Theorem \ref{th7_int},
their proof is thus omitted.
\vglue 0.2 cm
Concerning the null controllability, that is, the proof of Theorem \ref{null}, note that for the linear system we can get the result using the Carleman estimate provided by \cite[Propositon 3]{Gui} together with the following remark:
\begin{remark}
The following systems
\begin{equation}
\left\{\begin{array}[c]{lll}%
        u_t+u_x+u_{xxx}=f  &  & \text{ in } (0,L)\times(0,T),\\
        u_{xx}(0,t)=h_1(t),\text{ }u_x(L,t)=0,\text{ }u_{xx}(L,t)=0  &  & \text{ in }(0,T),\\
        u(x,0)=u_0(x) & & \text{ in }(0,L)
       \end{array}
\right.  \label{1a}
\end{equation}
and
\begin{equation}
\left\{\begin{array}[c]{lll}%
        y_t+y_x+y_{xxx}=f  &  & \text{ in } (0,L)\times(0,T),\\
        y(0,t)=k_1(t),\text{ }y_x(L,t)=0,\text{ }y_{xx}(L,t)=0  &  & \text{ in }(0,T),\\
        y(x,0)=y_0(x)  & & \text{ in }(0,L)
       \end{array}
\right.  \label{3a}
\end{equation}
are equivalent in the following sense:
\smallskip

For given $ \{u_0,f,h_1\} $ one can find $ \{y_0,f,k_1\} $ such that
the corresponding solution $ u $ of \eqref{1a} is exactly the same as the corresponding solution $ y $ for the
system \eqref{3a} and vice versa.
\medskip

Indeed, for given $u_0\in L^2(0,L)$, $f\in L^1(0,T;L^2(0,L))$ and
$h_1(t)\in H^{-\frac{1}{3}}(0,T)$, system \eqref{1a} admits a unique solution
$u\in C([0,T];L^2(0,L))\cap L^2(0,T;H^1(0,L))$. Let $ y_0 = u_0 $ and set $k_1(t)=h_1(t)$. Then, according to
Proposition \ref{p5}, we have $k_1(t)\in H^{\frac{1}{3}}(0,T)$. Due to the uniqueness of IBVP \eqref{3a}, with
the selection $\{y_0,f,k_1\}$, the corresponding solution $y\in C([0,T];L^2(0,L))\cap L^2(0,T;H^1(0,L))$
of \eqref{3a} must be equal to $u$, since $u$ also solves \eqref{3a} with the given auxiliary data $\{y_0,f,k_1\}$.
On the other hand, for any given $y_0\in L^2(0,L)$, $f\in L^1(0,T;L^2(0,L))$ and $k_1(t)\in H^{\frac{1}{3}}(0,T)$,
let $y\in C([0,T];L^2(0,L))\cap L^2(0,T;H^1(0,L))$ be the corresponding solution of the system \eqref{3a}.
From Proposition \ref{p5}, we have $y_{xx}(0,\cdot)\in H^{-\frac{1}{3}}(0,T)$. Thus, if $u_0=y_0$ and $h_1(t)=k_1(t)$,
then $h_1(t)\in H^{-\frac{1}{3}}(0,T)$ and the corresponding solution $u\in C([0,T];L^2(0,L))\cap L^2(0,T;H^1(0,L))$ of
\eqref{1a} must be equal to $y$, which also solves \eqref{1a} with the auxiliary data $\{u_0,f,h_1\}$.
\end{remark}
\begin{proof}[{\bf Proof of Theorem \ref{null}}:]
Consider $u$ and $\bar{u}$ fulfilling system \eqref{1abc} and \eqref{1ab},
respectively. Then $q=u-\bar{u}$ satisfies%
\begin{equation}
\left\{
\begin{array}
[c]{lll}%
q_{t}+q_{x}+( \frac{q^2}{2} +  \bar{u}q)_x  +q_{xxx}=0  &  & \text{in }(  0,L)  \times(  0,T)
\text{,}\\
q_{xx}(0,t)=h_1(t),\text{ }q_x(L,t) =0,\text{ }q_{xx}(L,t)  =0 &  &
\text{in }(  0,T)  \text{,}\\
q(x,0)  =q_{0}(  x)  :=u_{0}(  x)  -\bar
{u}_{0}(  x)  &  & \text{in }(  0,L)  \text{.}%
\end{array}
\right.  \label{ncn3}%
\end{equation}
The objective is to find $h_1$ such that the solution $q$ of (\ref{ncn3}) satisfies%
\[
q(\cdot,T)  =0\text{.}%
\]

Given $\xi\in \mathcal{Z}_T$ and $q_0:=u_0-\bar{u}_0\in L^2(0,L)$, we consider the following control problem
\ba
q_{t}+q_{x}+(\xi q)_x  +q_{xxx}=1_{\omega}
v(t,x)  &  & \text{in }(  0,L)  \times(  0,T)
\text{,} \label{CP1}\\
q_{xx}( 0,t)  =q_x(L,t)  =q_{xx}(L,t)  =0 &  &
\text{in }(  0,T)  \text{,} \label{CP2}\\
q(x,0)  =q_{0}(  x)   &  & \text{in }(  0,L)  \text{,} \label{CP3}%
\ea
where $v$ is solution of the following adjoint system
\begin{equation}
\left\{
\begin{array}
[c]{lll}
v_t+\xi(  t,x)v_x+v_{xxx}=0 &  &\text{in }(0,L)\times(0,T),\\
v(0,t)+v_{xx}(0,t)=0&  &\text{in }(0,T),\\
v(L,t)+v_{xx}(L,t)=0&  &\text{in }(0,T),\\
v_x(L,t)=0,&  &\text{in }(0,T),\\
v(x,0)=v_0(x), & &\text{in }(0,L).
\end{array}
\right.  \label{c7a}
\end{equation}
We can prove the following estimate
\begin{multline}
\label{estimCP}
||q||^2_{L^\infty(0,T,L^2(0,L))} + 2 ||q_x||^2_{L^2(0,T,L^2(0,L))}
\le \tilde C(T,L,||\xi||_{\mathcal{Z}_T})\big( ||q_0||^2_{L^2(0,L)}  + || v || ^2_{L^2( (0,T)\times \omega )} \big).
\end{multline}
We introduce the space%
\[
E:=C^0(  [0,T];L^2(  0,L)  )  \cap L^2(
0,T;H^1(  0,L)  )  \cap H^{1}(  0,T;H^{-2}(
0,L)  ),
\]
and in $L^{2}(  (  0,T)  \times(  0,L)  )  $ the
following set%
\[
B:=\left\{  z\in E; \ \left\Vert z\right\Vert _{E} \le 1  \right\} \text{.}%
\]
$B$ is compact in $L^2((0,T)\times (0,L))$, by Aubin-Lions's lemma. We will limit ourselves $v$ fulfilling the
condition
\begin{equation}
\label{bound}
|| v ||^2_{L^2((0,T)\times \omega )} \le C_* ||q_0||^2_{L^2(0,L) },
\end{equation}
where $C_\ast := C_\ast (T, L, || \bar u ||_{\mathcal{Z}_T}  + \frac{1}{2} )$. We associate with any $z\in B$, solutions of the linear system \eqref{1a}, the set%
\[
\begin{array}
[c]{c}%
T(  z)  :=\left\{  q\in B;\ \ \exists v\in L^{2}(  (0,T)  \times\omega)  \text{ such that }v \text{ satisfies } \eqref{bound}  \text{  and }\right. \\
\left.  q\text{ solves \eqref{CP1}-\eqref{CP3} with }\xi=  \bar{u} +\frac{z}{2}
\text{ and }q(\cdot,T)  =0 \right\} .
\end{array}
\]
By the result of the linear system (see \cite[Theorem 1]{Gui}) and (\ref{estimCP}), we see that if $\left\Vert
q_0\right\Vert _{L^2(  0,L)  }$  and $T$ are sufficiently small, then
$T(  z)  $ is nonempty for all $z\in B$. We shall use the
following version of Kakutani fixed point theorem (see e.g. \cite[Theorem
9.B]{Zeidler}):

\begin{theorem}
\label{fixed}Let $F$ be a locally convex space, let $B\subset F$ and let
$T:B\longrightarrow2^{B}$. Assume that

\begin{enumerate}
\item $B$ is a nonempty, compact, convex set;

\item $T(  z)  $ is a nonempty, closed, convex set for all $z\in B$;

\item The set-valued map $T:B\longrightarrow2^{B}$ is upper-semicontinuous;
i.e., for every closed subset $A$ of $F$, $T^{-1}(  A)  =\left\{
z\in B;\ T(  z)  \cap A\neq\varnothing\right\}  $ is closed.
\end{enumerate}

Then $T$ has a fixed point, i.e., there exists $z\in B$ such that $z\in
T(  z)  $.
\end{theorem}

Let us check that Theorem \ref{fixed} can be applied to $T$ and%
\[
F=L^{2}(  (  0,T)  \times(  0,L)  )  \text{.}%
\]
The convexity of $B$ and $T(  z)  $ for all $z\in B$ is clear. Thus
(1) is satisfied. For (2), it remains to check that $T(  z)  $ is
closed in $F$ for all $z\in B$. Pick any $z\in B$ and a sequence $\left\{
q^{k}\right\}  _{k\in\mathbb{N}}$ in $T(z)$ which converges in $F$ towards some
function $q\in B$. For each $k$, we can pick some control function $v^{k}\in
L^{2}(  (  0,T)  \times\omega)  $ fulfilling \eqref{bound}
such that \eqref{CP1}-\eqref{CP3} are satisfied with $\xi= \bar{u} +\frac{z}{2}  $
and $q^{k}(\cdot,T)  =0$. Extracting subsequences if needed, we
may assume that as $k\rightarrow\infty$%
\ba
v^{k}\rightarrow v &  & \text{in }L^{2}(  (  0,T)
\times\omega)  \text{ weakly,}  \label{t1} \\
q^k\rightarrow q &  & \text{in }L^{2}(  0,T;H^{1}(
0,L)  )  \cap H^{1}(  0,T;H^{-2}(  0,L)  )
\text{ weakly.}\label{t2}
\ea
By \eqref{t2}, the boundedness of $|| q^k ||_{ L^\infty (0,T,L^2(0,L))}$  and Aubin-Lions' lemma, $\{ q^k\}_{k\in \mathbb N}$ is relatively compact in $C^0([0,T],H^{-1}(0,L))$.
Extracting a subsequence if needed, we may assume that
\[
q^k\rightarrow q \text{ strongly in } C^0([0,T],H^{-1}(0,L)).
\]
In particular, $q(x,0)=q_0(x)$ and $q(x,T)=0$.  On the other hand, we infer from \eqref{t2} that
\[
\xi q^{k}\rightarrow\xi q\text{ in }L^{2}(  (  0,T)
\times(  0,L)  )  \text{ weakly.}%
\]
Therefore, $(\xi q^{k})_x \rightarrow (\xi q)_x$ in  ${\mathcal D}'(  (  0,T) \times(  0,L))$.
Finally,  it is clear that
\[ || v ||^2_{L^2((0,T)\times \omega )} \le C_* ||q_0||^2_{L^2(0,L)} \]
and that  $q$ satisfies \eqref{CP1} with $\xi=  \bar{u} +\frac{z}{2}$ and $q(\cdot,T)  =0$.
Thus $q\in T(  z)  $ and  $T(  z)  $ is
closed. Now, let us check (3). To prove that $T$ is upper-semicontinuous, consider
any closed subset $A$ of $F$ and any sequence $\left\{  z^{k}\right\}
_{k\in\mathbb{N}}$ in $B$ such that%
\begin{equation}
z^{k}\in T^{-1}(  A) , \quad \forall k\ge 0  \label{ncn6}%
\end{equation}
and
\begin{equation}
z^{k}\rightarrow z\ \text{ in }\ F, \label{ncn5}%
\end{equation}
for some $z\in B$. We aim to prove that $z\in T^{-1}(  A)  $. By
(\ref{ncn6}), we can pick a sequence $\left\{  q^{k}\right\}  _{k\in\mathbb{N}}$
in $B$ with $q^k\in T(z^k)\cap  A$  for all $k$, and a sequence $\left\{  v^{k}\right\}  _{k\in\mathbb{N}}$ in
$L^{2}(  (  0,T)  \times\omega)  $ such that%
\begin{equation}
\left\{
\begin{array}
[c]{lll}%
q_{t}^{k}+ q_x^k +  (  (  \bar{u}+\dfrac{z^{k}}{2})  q^{k})_{x}+q_{xxx}^{k}=1_{\omega}v^{k}(  t,x)  &  & \text{in }(
0,L)  \times(  0,T)  \text{,}\\
q^{k}_{xx}(0,t)  =q^{k}_x(L,t)  =q_{xx}^{k}(L,t)  =0 &  & \text{in }(  0,T)  \text{,}\\
q^{k}(x,0)  =q_{0}(  x)  &  & \text{in }( 0,L)  ,
\end{array}
\right.  \label{ncn7}%
\end{equation}%
\begin{equation}
q^{k}(x,T)  =0, \qquad \text{ in } (0,L) \label{ncn8}%
\end{equation}
and%
\begin{equation}
\left\Vert v^{k}\right\Vert ^2_{L^2(  (  0,T)  \times
\omega)  }\leq C_*\left\Vert q_{0}\right\Vert _{L^2(  0,L)} ^2. \label{ncn9}%
\end{equation}
From (\ref{ncn9}) and the fact that $z^{k}$, $q^{k}\in B$, extracting
subsequences if needed, we may assume that as $k\rightarrow\infty$,%
\[%
\begin{array}
[c]{lll}%
v^{k}\rightarrow v&  & \text{in }L^{2}(  (  0,T)
\times\omega)  \text{ weakly,}\\
q^{k}\rightarrow q &  & \text{in }L^{2}(  0,T;H^{1}( 0,L)  )  \cap H^{1}(  0,T;H^{-2}(  0,L)  )
\text{ weakly,}\\
q^{k}\rightarrow q & & \text{in } C^0([0,T],H^{-1}(0,L)) \text{ strongly},\\
q^{k}\rightarrow q &  & \text{in }F\text{ strongly,}\\
z^{k}\rightarrow z &  & \text{in }F\text{ strongly,}\\
\end{array}
\]
where $v\in L^2((0,T)\times \omega )$ and $q\in B$.
Again, $q(x,0)=q_0(x)$ and $q(x,T)=0$.  We also see that \eqref{CP2} and \eqref{bound} are satisfied. It remains to check that%
\begin{equation}
q_{t} +q_x +(  (  \bar{u} +\frac{z}{2})  q)  _{x}+q_{xxx}=1_{\omega } v (  t,x)  \text{.} \label{ncn10}%
\end{equation}
Observe that the only nontrivial convergence in (\ref{ncn7}) is those of the nonlinear term
$(  z^{k}q^{k})  _{x}$. Note first that
\[
||z^kq^k||_{L^2(0,T,L^2(0,L))} \le ||z^k || _{L^\infty (0,T,L^2(0,L))} ||q^k||_{L^2(0,T,L^\infty(0,L))} \le C,
\]
so that, extracting a subsequence, one can assume that $z^kq^k\rightarrow f$ weakly in $L^2((0,T)\times (0,L))$.
To prove that $f=zq$, it is sufficient to observe that
for any $\varphi\in {\mathcal D} (Q)$,
\[
\int_0^T\!\!\!\int_0^L z^kq^k\varphi dxdt\to \int_0^T\!\!\!\int_0^L z q\varphi dxdt,
\]
for $z^k\to z$ and $q^k\varphi \to q\varphi$ in $F$.
Thus%
\[%
\begin{array}
[c]{lll}%
z^{k}q^{k}\rightarrow zq &  & \text{in }L^{2}(  (  0,T)
\times(  0,L)  )  \text{ weakly.}%
\end{array}
\]
It follows that $(  z^{k}q^{k})  _{x}\rightarrow( zq)  _{x}$ in $\mathcal{D}^{\prime}(  (  0,T)
\times(  0,L)  )  $. Therefore, (\ref{ncn10}) holds and
$q\in T(  z)  $. On the other hand, $q\in A$, since $q^{k}\rightarrow q$ in $F$ and $A$ is closed. We conclude that $z\in T^{-1}(  A)  $,
and hence $T^{-1}(  A)  $ is closed.

Thus, follows from Theorem \ref{fixed} that there exists
$q\in  B$ with $q\in T(  q )$, i.e. we have found a control
$h_1\in L^{2}(0,T) $ such that the solution of
(\ref{ncn3}) satisfies $q(\cdot,T)  =0$ in $(  0,L)$. The proof of Theorem \ref{null} is finished.
\end{proof}

Finally we consider the boundary control system
\begin{equation}
\left\{\begin{array}[c]{lll}%
u_t+u_x+uu_x+u_{xxx}=0  &  & \text{ in } (0,L)\times(0,T),\\
u_{xx}(0,t)=h_1(t),\text{ }u_x(L,t)=h_2(t),\text{ }u_{xx}(L,t)=h_3(t)  &  & \text{ in }(0,T),\\
u(x,0)=u_0(x)  & & \text{ in }(0,L),
\end{array}
\right.  \label{ex-1}
\end{equation}
with all three control inputs being used and present the proofs of Theorem \ref{th6_int} and Theorem \ref{th6_int_z}.

\begin{proof}[{\bf Proof of Theorem \ref{th6_int}}:]
Consider the the initial value control of the KdV equation posed the the whole line $\mathbb{R}$:
\begin{equation}\label{ini-1}
w_t +w_x +ww_x +w_{xxx}=0, \quad w(x,0)= g(x)  \qquad x\in \mathbb{R}, \quad t\in  (0,T),
\end{equation}
where the initial value $g$ is considered as a control input.  By \cite[Theorem 1.2]{zhang2}  there exists a $\delta  >0$ such that, for $s\geq0$, if $u_0 , \ u_T \in H^s (0,L)$ satisfying
\[ \|u_0(\cdot)  -y(\cdot ,0)\|_{H^s (0,L)} +\|u_T(\cdot)-y(\cdot, T)\|_{H^s (0,L)} \leq \delta, \]
then one can choose $g\in H^s (\mathbb{R})$ so that (\ref{ini-1}) admits a solution
\[ w\in C([0,T]; H^s (\mathbb{R}) \cap L^2 (0,T; H^{s+1} (\mathbb{R}))\]
with
\[ w(x,0) =u_0(x), \qquad w(x,T)=u_T(x) \qquad for  \  x\in (0,L ).\]
Moreover, the solution $ w$ possesses the sharp Kato smoothing properties with
\[ h_1 (t):= w_{xx} (0,t)\in H^{\frac{s-1}{3}} (0,T), \qquad h_2 (t):= w_x (L,t)\in H^{\frac{s}{3}} (0,T), \qquad h_3 := w_{xx} (L,t) \in H^{\frac{s-1}{3}} (0,T) .\]
Thus  with such chosen control inputs $h_j, j=1,2,3$, $$u(x,t) := w(x,t), \quad \text{for} \quad x\in (0,L), \ t\in (0,T) $$
solves system (\ref{ex-1}) and satisfies
\[ u(x,0)=u_0 (x), \qquad u(x,T)= u_T (x), \qquad \text{for} \quad x\in (0,L).\]
We have, thus,  completed the proof of Theorem \ref{th6_int}.
\end{proof}
\noindent
\begin{proof}[{\bf Proof of Theorem \ref{th6_int_z}:}]
Without loss of generality, we assume $u_T=0$.  Consider feedback control system of the KdV equation posed on the interval $(-L, L)$:
\begin{equation}\label{ex-2}
\begin{cases}
v_t +v_x +vv_x +v_{xxx} +b(x) v=0  \quad v(x,0)=\tilde{u}_0 (x) & \quad x\in (-L,L), \ t\in (0,T),\\
v(-L,t)= 0, \quad v (L,t)=0, \qquad v_x (L,t) =0 & \quad t\in (0,T),
\end{cases}
\end{equation}
where
\[ b(x)= \left \{ \begin{array}{ll} 1& x\in (-\frac43 L, -\frac12 L),\\ 0 & \text{otherwise} \end{array}\right. \]
and
\[ \tilde{u}_0(x) =  \left \{ \begin{array}{ll}u_0 (x)& x\in (0,  L),\\ 0 & \text{otherwise}. \end{array}\right. \]
It follows from \cite{zhang2} that,  for given $u_0\in L^2 (0,L)$, we have $$v\in C_b (\mathbb{R}^+, L^2 (-L,L))\cap L^2_{loc} (\mathbb{R}^+; H^1 (-L,L))$$  and, there exists a $\nu >0$ such that
\[ \| v(\cdot, t)\|_{L^2 (-L,L)} \leq C\|u_0\|_{L^2 (0,L)} e^{-\nu t}, \quad \ \text{for any} \ t\geq 0 .\]
For given $\delta >0$,  choose $t^*$ large enough such that
\[ \| v(\cdot, t^*) \| _{L^2 (-L,L)} \leq C\|u_0\|_{L^2 (0,L)} e^{-\nu t^*} \leq \delta. \]
Then, again by \cite[Theorem 1.2]{zhang2}, one can find a control $h\in L^2 (t^*, t^*+1)$ such that  (\ref{1}) admits a solution $z\in C([t^*, t^* +1]; L^2 (0,L)\cap L^2 (t^*, t^* +1; H^1 (0,L))$  satisfying
\[ z(x,t^*)= v (x, t^*), \qquad z (x, t^*+1) = 0 \qquad x\in (0,L) .\]
Let $T=t^* +1$,
\[ h_1 (t) := \left \{ \begin{array}{ll} v_{xx} (0, t) & \quad t\in (0,t^*),\\ 0& \quad t\in (t^*, T), \end{array} \right. \qquad  h_2 (t) := \left \{ \begin{array}{ll} v_{x} (L, t) & \quad t\in (0,t^*),\\ h(t)& \quad t\in (t^*, T) \end{array} \right.\]
and
\[ \qquad  h_3 (t) := \left \{ \begin{array}{ll} v_{xx} (L, t) & \quad t\in (0,t^*),\\ 0& \quad t\in (t^*, T). \end{array} \right. \]
Note that as the solutions of (\ref{ex-2}) possess the sharp Kato smoothing properties  we have
\[ h_1 \in H^{-\frac13} (0,T), \quad h_2 \in L^2 (0,T), \quad h_3 \in H^{-\frac13} (0,T).\]
Thus if we let
\[ u(x,t) := \left \{ \begin{array}{ll} v(x,t) & \quad x\in (0,L), \quad t\in (0,t^*), \\  z(x,t) &\quad x\in (0,L), \quad t\in (t^*,T) , \end{array} \right.\]
then $u\in C([0,T]; L^2 (0,L)\cap L^2 (0,T;H^1 (0,L))$ solves (\ref{ex-1})  and satisfies
\[ u(x,0)= u_0 (x), \quad u(x, T) =0, \qquad \text{for} \ x\in (0,L).\]
Thus, the proof of Theorem \ref{th6_int_z} is achived.
\end{proof}

\noindent\textbf{Acknowledgments:}

Roberto A. Capistrano--Filho was supported by CNPq (Brazilian Technology Ministry), Project PDE, grant 229204/2013-9 and partially supported by CAPES (Brazilian Education Ministry).
Bing-Yu Zhang was supported by a grant from the Simons Foundation (201615), NSF of China (11231007) and PCSIRT (Chinese Education Ministry) under grant IRT 1273.

\end{document}